\documentclass[reqno,11pt,a4paper]{amsart}

\usepackage{amssymb}

\usepackage{enumitem}

\usepackage{color}
\definecolor{darkred}{rgb}{0.9,0.1,0.1}

\pdfpkmode={supre}
\usepackage{hyperref}
\hypersetup{
    colorlinks,
    citecolor=red,
    filecolor=black,
    linkcolor=blue,
    urlcolor=blue
}

\usepackage{fancyhdr} 

\newcommand{\changefont}{
    \fontsize{9}{11}\selectfont
}
\newtheorem{theorem}{Theorem}[section]

\newtheorem{proposition}[theorem]{Proposition}
\newtheorem{remark}[theorem]{Remark}
\newtheorem{lemma}[theorem]{Lemma}

\theoremstyle{definition}

\newcommand{\mb}[1]{\mathbb{#1}}
\newcommand{\mc}[1]{\mathcal{#1}}

\newcommand{\mbf}[1]{\mathbf{#1}}

\newcommand{\ra}{\rightarrow}
\newcommand{\lp}{\langle}
\newcommand{\rp}{\rangle}

\newcommand{\vp}{\varphi}
\newcommand{\ve}{\varepsilon}

\newcommand{\im}{\text{i}}

\DeclareMathOperator{\hess}{Hess}
\DeclareMathOperator{\supp}{supp}

\author[M. Brooks]{Morris Brooks}
\address{Morris Brooks, 
IST Austria,
Am Campus 1,
3400 Klosterneuburg,
Austria}
\email{morris.brooks@ist.ac.at}

\author[G. Di Ges\`u]{Giacomo Di Ges\`u}
\address{ Giacomo Di Ges\`u, TU Vienna, 
Wiedner Hauptstr.\ 8, 
1040 Vienna, Austria.}
\email{giacomo.di.gesu@tuwien.ac.at}

\RequirePackage{color}

\title{Sharp tunneling estimates for a double-well model in infinite dimension}

\begin{document}


\begin{abstract}
We consider the stochastic quantization of a quartic double-well 
energy functional in the semiclassical regime and derive optimal asymptotics for 
 the exponentially small splitting of the ground state energy.
Our result provides an infinite-dimensional
version of 
some sharp tunneling estimates 
known in finite dimensions for 
 semiclassical Witten Laplacians in degree zero.
 From a stochastic point of view it 
proves that the
$L^2$ spectral gap of the stochastic one-dimensional Allen-Cahn equation in finite volume
satisifies a Kramers-type formula in the limit of vanishing noise. 
We work with finite-dimensional lattice approximations and establish semiclassical estimates which are uniform in the dimension. Our key estimate shows 
that the constant separating the two exponentially small eigenvalues from the rest of the spectrum can be taken independently of the dimension.

\end{abstract}

\maketitle
\section{Introduction}

\noindent
The study of the semiclassical eigenvalue 
splitting due to tunneling effects in multiwell systems 
has a long 
history dating back to the beginnings
 of quantum mechanics. In
 the original setting one deals 
with the
Schr\"odinger operator  in finite dimensions
\begin{equation}\label{schrodinger}   H_h   =   - h^2 \Delta    + V,    
\end{equation}
and with the semiclassical approximation $h\to 0$, describing
the transition from quantum to classical mechanics~\cite{HLNM, Zworski}.
If $V: \mb R^N\to \mb R$ is a confining symmetric double-well potential,
 the difference $E_h^{(1)}-E_h^{(0)}$  between the two smallest eigenvalues  
  of $H_h$ turns out to be exponentially small in $h$, and the
  exponential decay rate of this eigenvalue splitting is determined by the 
  so-called Agmon distance~\cite{Si84, HSI, HLNM}.

Besides the original motivation of the semiclassical approximation 
for quantum systems, the analysis of the eigenvalue splitting for Schr\"odinger operators has been proven to be fruitful in a large number of different situations. 
These include 
problems in 
statistical mechanics, 
following Kac's early ideas~\cite{Kac} on eigenvalue degeneracy as ultimate characteristic of first order phase transitions; and in particular 
the problem of metastability, an example of a dynamical phase transition
\cite{FW,OV, BdH}. 
 Other applications can be found in Differential Topology, more specifically regarding Morse Homology~\cite{Bott}, following the pioneering paper of 
Witten~\cite{Witten}. 

In both types of applications, a multiwell potential $V$ naturally appears.
However, the relevant 
Schr\"odinger operator turns out to be not~\eqref{schrodinger}, but rather of the form 
\begin{equation}  \label{introWitt}
\tilde H_h =   - h^2 \Delta   +      W_h,       \text{ where }    W_h   :=  \tfrac 14|\nabla V|^2 - 
\tfrac h2 \Delta V.  
\end{equation}
This operator is sometimes referred to as Witten Laplacian or, more precisely, Witten Laplacian in degree zero, since it is the restriction on the level of $0$-forms (i.e functions) of 
the full-fledged Witten Laplacian acting on the exterior algebra of differential forms. 
The peculiar form of $\tilde H_h$ might be best understood by observing that, up to a factor $h$, it is unitarily equivalent via ground state transformation to the diffusion operator
\begin{equation}   \label{introGen}   L_h     =   - h \Delta   +    \nabla V \cdot \nabla.
\end{equation}
The latter acts as selfadjoint operator on the weighted space $L^2(e^{-V/h}dx)$ and its  
associated quadratic form is given by  
\begin{equation}   \label{introDir}    \mc E_h [f]=   h   \int_{\mb R^N}   |\nabla f|^2 e^{- V/h} dx     .  
\end{equation}
Moreover $L_h$ is 
the $L^2$-generator of the stochastic Langevin dynamics
\begin{equation}   \label{introLang}      \dot X     =   -  \nabla V (X)     +    \sqrt{2 h} \, \eta     ,    
\end{equation}
where $t\mapsto X(t)$ is a stochastic process in $\mb R^N$ and $\eta$ is an $N$-dimensional white noise in time.  From this stochastic point of view the semiclassical asymptotic $h\to 0$  turns out to be a small noise asymptotic for a reversible diffusion process. 
The procedure which, starting from $V$, constructs the 
essentially 
equivalent objects~\eqref{introWitt}-\eqref{introLang} 
and thus establishes a connection between the formalism of Quantum mechanics and diffusion processes is also called stochastic quantization~\cite{Nelson, AlbeverioHK, Pa, JLMS}.

\

There exists a large literature concerning the semiclassical eigenvalue splitting for the stochastic quantization
~\eqref{introWitt}-\eqref{introLang}, see e.g.~\cite{FW, Davies3, HKS, KoMa, Mathieu, Miclo} and references therein.
Sharpest possible results have been obtained on the asymptotic behaviour of all the exponentially small eigenvalues 
in the case of rather  general multiwell potentials in~\cite{HKN,BGK,Eckhoff,Su, MeSc}.
For example, in the case of a potential $V$ with nondegenerate critical points, exactly two quadratic minima  and growing sufficiently at infinity, the two smallest eigenvalues $E_h^{(0)}, E_h^{(1)}$ of~\eqref{introWitt}
satisfy  $E^{(0)}_h = 0$ and, in the limit $h\to 0$,
\begin{equation}  \label{IntroSG}    E^{(1)}_h =    h A   \exp(  -      B/ h    )      \,  \left(  1  +    o(1)  \right)  .          \end{equation}
Here $A, B >0$ are constants which can be computed explicitly from the potential $V$. Note that $  \lambda_1(h) := \tfrac{E_h^{(1)}}{h}$ is nothing but the $L^2$-spectral
gap of~\eqref{introLang}. 

\

\noindent
{\bf Semiclassical tunneling in infinite dimensions.}
This paper 
concerns the problem of obtaining sharp asymptotics of the type~\eqref{IntroSG} in infinite-dimensional models. 
In the presence of spatially extended systems with infinite degrees of freedom, 
as  occuring
in statistical mechanics and quantum field theory, 
the underlying finite-dimensional manifold of states is replaced by a suitable infinite-dimensional topological space. 
A typical energy functional for a  system described by a field $\xi: 
\Lambda \to \mb R$ then takes the form 
\begin{equation} \label{IntroV}    V(\xi)    =   \int_{\Lambda}  F(\xi(s)) \, ds    +      \tfrac J2   \int_{\Lambda}    |\nabla \xi(s)|^2  \, ds  ,              
\end{equation}
where $\Lambda$ is some region in $\mb R^d$, $J>0$ is a constant and $F:
\mb R \to \mb R$ is a local potential. Analogous functionals appear in topological
applications when considering e.g. infinite-dimensional Riemannian manifolds of loops, as Witten already had in mind (see Section 4 of~\cite{Witten} and e.g.~\cite{Aida03c, Eberle}).

In most situations of interest the energy landscape determined by $V$ is rather complex and in particular $V$ might have several distinct local minima. In analogy to the finite-dimensional case one expects then that 
 exponential eigenvalue splitting occurs for the corresponding stochastic quantization 
 of $V$ in the semiclassical regime.  
 
  Our aim 
 is to put forward a general strategy 
 for extending~\eqref{IntroSG} to infinite-dimensional situations.
 We illustrate this strategy by giving a  complete proof of~\eqref{IntroSG}
 for a special and relatively simple instance of~\eqref{IntroV}, where $\Lambda$ is the one-dimensional torus and $F$ is a symmetric quartic double-well. More specifically we restrict to the case
 \[    F(\xi)     =     \tfrac 14   \xi^4   - \tfrac 12 \xi^2      . \]
 For simplicity we also assume that $J$ is large enough (sepecifically $J> \tfrac{4}{\pi^2}$), so that 
 $V$ admits exactly two minima, given by the constant states $\pm 1$.
 The resulting double-well functional $V$ is sometimes referred to as Ginzburg-Landau or Allen-Cahn energy functional. While infinite-dimensional versions of the Schr\"odinger Operator~\eqref{introWitt} are generally ill-defined, it is well-known that mathematically sound interpretations of ~\eqref{introGen}-\eqref{introLang} can be given in the case of the Ginzburg-Landau functional considered here, see e.g.~\cite{DP}.  In particular 
 the 
 Langevin dynamics~\eqref{introLang}    
 becomes now the semilinear stochastic partial 
 differential equation 
 \begin{equation*}\label{}
   \partial_t u    =   J\, \partial_x^2 u   -   u^3   + u    + \sqrt{2h} \, \eta   ,  
 \end{equation*}
 where $t\mapsto u(t)$ is a stochastic process taking values in a space of functions on $\mb T$ and 
 $\eta$ is a space-time white noise~\cite{FJL}.

Our main result, 
 Theorem~\ref{corollaryEK} below, states that
 for this infinite-dimensional quartic model the asymptotic relation  ~\eqref{IntroSG} holds true with
 $B =\tfrac 14$ (the height of the barrier separating the wells), an explicit prefactor  
 $A$, expressed for notational convenience in terms of $\mu:= \tfrac{\pi^2 J}{4}$,
 and with the $o(1)$ remainder term of order $\mc O(h)$. We emphasize that the difficult part of this result concerns the lower bound. The upper bound 
 follows indeed rather easily from a suitable choice of test functions already introduced in~\cite{DGLP}.

 \
 
\noindent
{\bf Previous results.} Several studies have been devoted to questions of semiclassical analysis in large and infinite dimension. 
Early contributions are~\cite{Sj1, Sj2, DobKol} and  also~\cite{MaMo}
 for semiclassical estimates  with uniform bounds on the dimension, mainly restricted to one-well situations. Further we mention Aida's extensive work on infinite-dimensional Schr\"odinger operators, see~\cite{AidaSug} for an overview, and in particular his paper~\cite{Aida12} on semiclassical tunneling. The latter concerns an infinite-dimensional version of~\eqref{schrodinger} with a renormalized polynomial potential $V$ and shows one-sided estimates on the exponential decay rate in terms of a suitable Agmon estimate.  
 
\

To our knowledge no rigorous results 
 on spectral asymptotics comparable in precision with~\eqref{IntroSG} have been established so far even in simple double-well situations. 
There exist at least three possible methods
to show~\eqref{IntroSG} in finite dimension:
1) the potential-theoretic approach which goes through the computation of hitting times for the stochastic dynamics~\eqref{introLang}
and exploits variational principles for capacities~\cite{BdH}; 
2) the semiclassical approach \`a la Helffer-Klein-Nier  based on 
 WKB expansions  and supersymmetry arguments~\cite{HKN}; 
3) the approach in~\cite{MeSc} exploiting variance decompositions and optimal transport techniques.

\

The first method has been used also in infinite-dimensional settings, both for generalizations of the quartic model over $\mb T$ treated here~\cite{BaBo, Ba, BG} 
and for the corresponding renormalized problem over the two-dimensional torus $\mb T^2$~\cite{BDW}. These papers yield a result for the average time  needed for the stochastic process to pass from one well to the other. A result 
of this type is commonly called Kramers' Law and comparable in precision with~\eqref{IntroSG}. The deduction of sharp eigenvalue asymptotics from  an infinite-dimensional Kramers Law is however 
missing.

A first attempt to generalize approach 2) to the infinite-dimensional model considered here was made in~\cite{DGLP}. 
The authors consider finite-dimensional lattice approximations and 
show that~~\eqref{IntroSG} holds uniformly in the dimension if the two exponentially small eigenvalues are well separated from the rest of the spectrum \cite[Theorem 1.2]{DGLP}. We shall refer to the latter property 
as a rough spectral separation.
As explained in~\cite{DGLP} the usual finite-dimensional estimates 
fail to produce 
a rough spectral separation that holds uniformly in the dimension. This is indeed the major issue when trying to lift tunneling calculations to infinite dimensions. 
A very similar problem arises with approach 3): 
here, a rough spectral separation for each basin of attraction is needed,
and the method based on Lyapounov functionals employed 
in~\cite{MeSc} does not produce uniform estimates.

\

\noindent
{\bf Methods.} The main technical contribution of this paper is to solve the 
above mentioned problem and to show that  the rough spectral separation holds uniformly in the dimension
(see Theorem~\ref{thMAIN} below).

One part of the proof consists in a suitable landscape decomposition and 
a careful choice of reference potentials for the localized problems. These lead via ground state transformations to infinite-dimensional Schrödinger operators replacing the ill-defined
infinite-dimensional version of~\eqref{introWitt}. A second part provides  the 
estimates for the localized problems. The main ingredient here is the NGS bound~\cite{Gross, Si_EQFT}, which is well-known from quantum field theory and has already been used in~\cite{Aida03b}
in a semiclassical context, see also~\cite{FGZ}. 
It provides a quantitative operator bound in case of a singular Schrödinger potential.
The NGS bound follows from (and is indeed equivalent to)
the regularizing effect of a logarithmic Sobolev inequality.
To obtain the necessary hypercontractive properties we exploit the convexity poperties of the reference potentials and the Bakry-\'Emery criterion. 

All our computations are performed for finite-dimensional lattice approximations with uniform estimates in the approximation. 
The analogous computations could, however, also be done
directly in infinite dimension. Our approach has the advantage 
to include also uniform estimates for the lattice approximation. The latter, we believe, has its own interest, both as a physical model and for numerical schemes. 
Moreover our approach substantially reduces the technical prerequisites for the proof: the infinite-dimensional objects enter only in the final limiting procedure.

\

\noindent
{\bf Outline of the paper.} In Section~\ref{sec.Results} we
present our main results:
the crucial estimate is stated in
Theorem~\ref{thMAIN}, which  
provides the rough spectral separation for the lattice approximation of the model.  Theorems~\ref{thKramers} and~\ref{corollaryEK} provide the sharp spectral gap asymptotics respectively for the lattice approximation and the infinite-dimensional model. While the former is a direct consequence of~Theorem~\ref{thMAIN}, for the latter some additional approximation results are needed. These are straigthforward and are discussed in the final Section~\ref{sec.Inf} for completeness.   The core of the paper is given by
the Sections~\ref{SectionStrategy}-\ref{sec.Loc}, where we prove 
Theorem~\ref{thMAIN}. More precisely, in Section~\ref{SectionStrategy} we
 reduce the problem to localized problems around and off the diagonal. 
In Section~\ref{sec.Aux} we recall some abstract auxilary tools, in particular the NGS Bound, in the form needed for the subsequent analysis of the localized  problems. In Section~\ref{sec.Loc} we prove all the necessary local estimates which permit to conclude the proof of Theorem~\ref{thMAIN}.

\section{Results}\label{sec.Results}

\noindent
As in~\cite{DGLP} we fix $\mu>1$ and consider for every dimension $N\in \mb N$
the function $V_N:\mb R^N  \ra  [0, \infty)$ defined by 
\begin{equation}\label{defV}     V_N(x)      :=           \sum_{k=1}^N    \frac 14\big(     x_{k}^2  -1       \big)^2      +          \frac{\mu}{8\sin^2(\frac{\pi}{N})}   \   \sum_{k=1}^N  
   (x_{k} - x_{k+1})^2      ,    
   \end{equation}
where $x_{N+1}:=x_1$. 
We shall refer to $V_N$ as the energy. It is straightforward to show 
 (see e.g.~\cite[Lemma 2.1]{DGLP}) that, due to  the assumption $\mu>1$, for every $N\in \mb N$ 
the function $V_N$ admits exactly three critical
points: the two global minimum points given by the constant states $I_+ =(1,\dots,1 )$
and $I_- =(-1,\dots,-1)$ and the critical point of index one given by the origin $O=(0, \dots, 0)$. One might think of $\tfrac 1 N  V_N $ as a lattice approximation of the double-well functional $V : H^1(\mb T) \to [0,\infty)$ defined by 
\[    V (\xi)    :=    \int_{\mb T}   
  \frac 14\big(    \xi^2(s)  -1       \big)^2 ds      +        
 \frac{\mu}{8 \pi^2} \int_{\mb T}   |\xi'(s)|^2 ds  , \]
 where $\mb T$ is the one-dimensional torus $\mb R/\mb Z$. 

\

\noindent
Let $C_b^\infty(\mb R^N)$ be the space of smooth real functions on $\mb R^N$ which are bounded together with all their derivatives. 
 For each $N\in \mb N$ and $h>0$ we denote by  $\mc E_{h, N}: C_b^\infty(\mb R^N)\to \mb R$
the quadratic form defined by  
\[     \mc E_{h,N} [f]       :  =      hN   \,   \int_{\mb R^N}  |\nabla f (x)|^2  \, 
e^{- \frac{V_N (x)}{hN}} \, dx     ,     \]
define for each finite-dimensional linear subspace $S\subset C_b^\infty(\mb R^N)$ 
\begin{equation}\label{definitionkappa}  
  \kappa^{(S)}_{h,N}     :=  
\sup \left\{  \mc E_{h,N} [f] : f \in S \text{ and } \int_{\mb R^N} f^2  
 e^{- \frac{V_N (x)}{hN}} dx    =   1   
   \right\}    \   , 
   \end{equation}
and finally consider for each $j\in \mb N_0$  
\begin{equation}\label{minmaxdef}  
   \lambda^{(j)}_{h, N}   \  :=  \  
    \inf \left\{      \kappa^{(S)}_{h,N} :   S\subset  C_b^\infty(\mb R^N)  \text{ and }   \dim S = j+1     \right\}       .       
    \end{equation}
 It follows from standard arguments that for each $h>0, N\in \mb N$ the set 
  $\{\lambda^{(j)}_{h, N} \}_{j\in \mb N_0}$ gives counting multiplicities both the spectrum of the
 closure in $L^2(e^{-V_N/hN}dx)$ of the diffusion-type differential operator 
 \begin{equation*} \label{defgenerator}    f \mapsto    - hN \Delta f  +     \nabla V_N \cdot \nabla f ,    \   \   \       f\in C_b^\infty(\mb R^N)    , 
 \end{equation*}
 and the spectrum of the
 closure in $L^2(dx)$ of the Schr\"odinger-type differential operator
 \begin{equation} \label{defschrodinger}      f\mapsto   - hN    \Delta  f  +   ( \tfrac{1}{4hN} |\nabla V_N|^2   - \tfrac 12 \Delta V_N)  f       ,   \  \  \         f\in C_c^\infty(\mb R^N)    , 
 \end{equation}
 where $C_c^\infty(\mb R^N)$ is the space of smooth compactly supported real functions on $\mb R^N$. The differential operator~\eqref{defschrodinger} is also known as (the restricion on $0$-forms of) the Witten Laplacian corresponding to the energy $V_N$~\cite{Witten}. 

\

\noindent
Note that   $ \lambda^{(0)}_{h, N} =0$
and $ \lambda^{(1)}_{h, N} >0$
for each $h>0$ and $N\in \mb N$.
Moreover, considering a suitable test function for the upper bound and rough perturbation arguments for the lower bound, one can show that 
$\lambda^{(1)}_{h, N}$ is exponentially small in the regime $h\ll 1$, uniformly in the dimension $N$. More precisely \cite{DGLP} there exist
$C, C'  >0$ such that for all $h\in (0, 1]$ we have 
\[       e^{- \tfrac{C}{h}}      \leq      \lambda^{(1)}_{h, N}         \   \leq    \     e^{- \tfrac {C'}{h}}       .   \]

\noindent
The main technical result of the present paper is to prove that $\lambda^{(2)}_{h, N} $ is bounded from below, uniformly both in $h$ and $N$. 
\begin{theorem}\label{thMAIN}
There exist constants $C, h_0>0$ such that for every $h\in(0, h_0]$ and every $N\in \mb N$  
we have
\[   \lambda^{(2)}_{h,N}    \   \geq   \   C     .   \]
\end{theorem}

\noindent
Thus, in the semiclassical regime $h\to 0$, the spectrum separates sharply 
into a ``low-lying spectrum''  consisting of the two exponentially close eigenvalues $ \lambda^{(0)}_{h,N},  \lambda^{(1)}_{h,N}$
and the rest of the spectrum, which is uniformly bounded from below by a
strictly positive constant. 
A statement of this type is well known to hold for a general class of  energies $V$ when the dimension $N$ is fixed. In general one finds indeed a cluster $m_0$ 
of exponentially small eigenvalues, where $m_0$ is the number of local minima of $V$; then there is 
a large gap, with  the rest of the spectrum being bounded away from zero, uniformly in $h$. However the usual arguments,
based on suitable Harmonic approximations of the 
Schr\"odinger operator~\eqref{defschrodinger}~\cite{Si83, HSI, CFKS, KRI}, 
do not permit to get bounds uniform in $N$ 
when applied to the sequence of quartic energies $V_N$ defined in~\eqref{defV}.

\

\noindent
As a corollary of Theorem~\ref{thMAIN} we are able to compute the precise asymptotic behaviour in the limit $h\to 0$ of the spectral gap $ \lambda^{(1)}_{h, N} $, with uniform control in the dimension $N$: 
The 
exponential decay rate of $ \lambda^{(1)}_{h, N} $ equals the height 
of the barrier $\tfrac 1N V_N(0) - \tfrac 1N V_N(I_+) = \tfrac 14$
separating the two wells, 
with an explicit $h$-independent pre-exponential factor given by 
\begin{equation}\label{prefactor}
p(N)   \ :=    \ \frac 1\pi   \  \left|\frac{\det\hess V_N(I_+)}{\det\hess V_N(0)}\right|^{\frac12}. 
 \end{equation}
 More precisely we obtain as immediate application of 
 Theorem~\ref{thMAIN}  and ~\cite[Theorem 1.2]{DGLP}
 the following uniform Kramers Law, which, together with its infinite-dimensional version stated below, is the main result of our paper.  
\begin{theorem} \label{thKramers} Let $p(N)$ be given by~\eqref{prefactor}. Then there exist $h_0, C>0$ such that the error 
term  $(h, N)  \mapsto \epsilon (h, N)$ defined for $h>0$ and $N\in \mb N$ by 
 \begin{equation*}
 \lambda^{(1)}_{h, N} 
  =   
    p(N)    
e^{-\frac{1}{4h}}
\  \big(    1   +       \epsilon(h,N)      \big)          ,
\end{equation*} satisfies for all $h\in (0, h_0]$ and $N\in \mb N$
the bound    $   |\epsilon(h,N)|       \leq        C    h $.  
\end{theorem}
\noindent
As already noted in~\cite{St} for $N\to\infty$ the prefactor converges: 
 \begin{equation}\label{prefactorconv}p(N) \to  \tfrac{ \sinh (\pi \sqrt{2\mu^{-1}})}{ \pi  \sin (\pi \sqrt{\mu^{-1}})} .   
 \end{equation}
Since our bounds are uniform in $N$ we can thus pass to the limit 
$N\to \infty$ and get the corresponding infinite-dimensional version 
of Theorem~\ref{thKramers}.
To formulate the latter, we shall introduce the following notation.  
We fix a mass $m>0$ and consider the trace class operator 
$\mc G:L^2(\mb T) \to L^2(\mb T)$ defined as the inverse of the 
selfadjoint operator  on $L^2(\mb T)$ given by 
$H^2(\mb T) \ni x\mapsto mx - \tfrac{\mu}{(2\pi)^2}x''$. We 
denote by $\gamma_{h}$ the centered Gaussian measure on $ L^2(\mb T)$ with covariance operator  $h \mc G$ and define 
$U:L^2(\mb T) \to \mb R \cup\{+\infty\}$ by 
\begin{equation}\label{defUc}  U(\xi)    :=    \int_{\mb T}   
  \big(  \frac 14  \xi^4(s)  -    \frac{1}{2}(m+1)   \xi^2(s)       \big) ds    +    \frac 14 .  
  \end{equation}
  Finally we define for every $h>0$
\[    \lambda^{(1)}_h   :=   \inf_{F}
\Big\{    h  \int_{L^2(\mb T))}  \|D F \|_{L^2(\mb T)}^2  \, 
e^{- \frac{U }{h}} \, d\gamma_h   \Big\}    , \]
where the infimum is taken over all  
$F\in \mc FC_b^\infty(L^2(\mb T))$ satisfying
the constraints $ \int_{L^2(\mb T))}  | F |^2  \, 
e^{- \frac{U }{h}} \, d\gamma_h =1 $
and   $\int_{L^2(\mb T))}   F   \, 
e^{- \frac{U }{h}} \, d\gamma_h = 0$ and where
$DF$ is the gradient of $F$ . Here 
$F\in \mc FC_b^\infty(L^2(\mb T))$ means that 
$F$ is a cylindrical test function on $L^2(\mb T)$, i.e. there exist
$n\in \mb N$, $y_1, \dots, y_n\in L^2(\mb T)$ and $f\in C_b^\infty(\mb R^n)$ such that $F(x) =   f( \lp x, y_1\rp, \dots, \lp x, y_n\rp)$ for every $x\in L^2(\mb T)$.

\begin{theorem}\label{corollaryEK}  There exist $h_0, C>0$ such that the error 
term  $h  \mapsto \epsilon (h)$ defined for $h>0$  by 
 \begin{equation*}
 \lambda^{(1)}_{h} 
  =   
    \tfrac{ \sinh (\pi \sqrt{2\mu^{-1}})}{ \pi  \sin (\pi \sqrt{\mu^{-1}})}  \    
e^{-\frac{1}{4h}}
\  \big(    1   +       \epsilon(h)      \big)       ,
\end{equation*} satisfies for all $h\in (0, h_0]$ 
the bound    $   |\epsilon(h)|       \leq        C    h $.  
\end{theorem}
\noindent
Analogously to the finite-dimensional case, we have 
now that  $\lambda^{(1)}_h$ equals 
 the smallest non-zero eigenvalue of the closure in $L^2(e^{-U/h}\gamma_h )$ of the infinite-dimensional diffusion-type differential operator 
 \begin{equation} \label{defgeneratorinfty}    F \mapsto    -  L_h F  +    \lp D U , D F\rp_{L^2(\mb T)}    \  \   \      ,    \   \   \       F\in \mc FC_b^\infty(L^2(\mb T))    , 
 \end{equation}
 where $L_h$ denotes the Ornstein-Uhlenbeck operator which has $\gamma_h$ as invariant measure. 
Note that, after suitable unitary transformation, one might think of~\eqref{defgeneratorinfty} also as a rigorous version of 
the infinite-dimensional Schr\"odinger-type differential operator
in $L^2(\gamma_h)$
formally given by 
 \begin{equation} \label{defschrodingerinfty}      
 F\mapsto   - L_h F +    (\tfrac{1}{4h} \|D U\|_{L^2(\mb T)}^2  - \tfrac {1}{2h} L_h U) F   . 
 \end{equation}
We remark explicitly that~\eqref{defgeneratorinfty} is the $L^2$-generator of the following nonlinear stochastic heat equation, which is known under various names as e.g. stochastic Allen-Cahn, or Chafee-Infante equation: 
\[    \partial_t u    =   \tfrac{\mu}{4\pi^2} \partial_x^2 u   -   u^3   + u    + \sqrt{2h} \xi   .   \]
 On the other hand the operator~\eqref{defschrodingerinfty}  
 might be seen as an infinite-dimensional Witten Laplacian 
 (restricted to $0$-forms).

\section{Reduction to local problems around and off the diagonal}
 \label{SectionStrategy}

\noindent
We denote by $ \overline x:= \frac 1N \sum_{k=1}^N x_k  $
 the average of $x\in \mb R^N$.
The first step of the proof 
of Theorem~\ref{thMAIN} consists in decomposing the problem into two pieces: one localized in a small neighbourhood
of the space of constant states, i.e. satisfying $x- \overline x=0$, and one in the 
complementary set, which will turn out to be negligible for the low-lying spectrum in the $h\to 0$ limit.  
Since we want to analyze a quadratic form, the decomposition is most conveniently realized via a smooth quadratic partition of unity:

\

\noindent
We fix a $\chi\in C_{{\rm c}}^{\infty}(\mb R;[0,1])$
such that  $\chi\equiv1$ in $[- \tfrac 12 , \tfrac 12]$ and $\chi \equiv 0$
in $\mb R\setminus [- 1, 1]$. Further, for each $N\in \mb N$
and $R>0$ we define    $\theta_{N,R}$, $\tilde \theta_{N,R}:\mb R^N \ra [0,1]$ by setting  
\begin{equation}\label{defkappa0}
    \theta_{N, R} (x)   : =        \chi \Big( \tfrac{1}{N R^2} \sum_k (x_k-\overline x)^2\Big) 
\    \   ,      \     \   \tilde\theta_{N,R}(x)    :=   
   \big(1-\theta_{N,R}^2 (x)\big)^\frac12    ,  
 \end{equation}
 so that $ \theta^2_{N, R} +  \tilde \theta^2_{N, R}  \equiv 1$. 
Note that $\theta_{N,R}, \tilde \theta_{N,R}\in C^\infty(\mb R^N)$ 
for all $N\in \mb N$, $R>0$. 

\

\noindent
From a straightforward computation of commutators 
one gets for all $N\in \mb N$, $h,R>0$ and all $f\in C_b^\infty(\mb R^N)$
the identity (in general also known as IMS localization formula, see~\cite[Theorem 3.2]{CFKS})
\begin{gather}\label{firstIMSdecomp}
\mc E_{h, N} [f]  
=  \mc E_{h, N} [\theta_{N,R} f]    +   
\mc E_{h, N} [\tilde \theta_{N,R} f]
         +        h \mc F_{h, N, R}[f]           ,   
\end{gather}
where the localization error $\mc F_{h, N, R}[f]$ is given by 
 \[ \mc F_{h, N, R}[f] :=  -   N \int_{\mb R^N} \left( |\nabla \theta_{N,R}|^2  +  |\nabla \tilde\theta_{N,R}|^2    \right)  f^2  \, e^{- \frac{V_N }{hN}}  dx 
  .    \]
 Since for each $R>0$ there exists a constant $c(R)$ such that for every $N\in \mb N$
\begin{equation}\label{cutoffgradients}     N \left( |\nabla \theta_{N,R}|^2  +  |\nabla \tilde\theta_{N,R}|^2    
\right) \leq c(R)   ,  
\end{equation}
one obtains immediately for all $N\in \mb N$, $h,R>0$ and $f\in C_b^\infty(\mb R^N)$ 
\begin{equation}\label{Boundonlocalizationerror}    |\mc F_{h, N, R}[f] |  \leq  c(R)  \int_{\mb R^N}   f^2  \, e^{- \frac{V_N }{hN}}  dx .   
\end{equation}
As we show below at the end of this section, Theorem~\ref{thMAIN} will then be an easy consequence of the following two propositions. 
The first one concerns the term $ \mc E_{h, N} [\tilde\theta_{N,R} f] $
on the right hand side of~\eqref{firstIMSdecomp}. Indeed it implies that away from the diagonal the quadratic form $\mc E_{h,N}$ is large in $h$ and therefore does not contribute to the low-lying spectrum.

\begin{proposition}\label{MainPropAway}
 For every $R>0$ 
there exist constants $ C=  C(R)>0$ and   $h_0=  h_0(R)>0$ such 
that for all $h\in(0,  h_0]$, $N\in \mb N$ 
and  $ f\in C_b^\infty(\mb R^N)$   with 
$    \supp f 
 \subset
   \Big\{x\in \mb R^N:     \tfrac 1N \sum_{k} (x_k - \overline x)^2 \geq  R^2   \Big\}$ we have
\[ \mc E_{h,N} [f]     \geq      h^{-1} C    \int_{\mb R^N}  f^2 e^{-\tfrac{V_N}{hN}}  dx   . 
     \]
\end{proposition}
\noindent
The second proposition concerns the term 
 $\mc E_{h, N} [\theta_{N,R} f]$
 on the right hand side of~\eqref{firstIMSdecomp}.
 It states that, when restricting to a sufficienlty small neighbourhood of the diagonal, $\mc E_{h,N}$ is larger than a constant except on a linear subspace of dimension at most $2$. 
\begin{proposition}\label{MainPropDiag}
There exist constants $R_0, h_0, C>0$ 
and, for every 
$h>0, N\in \mb N$, there exist functions $\phi^{+}_{h, N} ,\phi^{-}_{h, N}  \in C_b^\infty(\mb R^N)$ such that 
for all $h\in (0, h_0]$, $N\in \mb N$ and 
  $f\in C_b^\infty(\mb R^N)$ with    
  $\supp f  \subset  \Big\{x\in \mb R^N:     \tfrac 1N \sum_{k} (x_k - \overline x)^2 \leq  R_0^2   \Big\} $
  we have 
\begin{equation}\label{QLBdiag} \mc E_{h,N} [f]     \geq      C    \int_{\mb R^N}  f^2 e^{-\tfrac{V_N}{hN}}  dx  
     -    \Big(\int_{\mb R^N}  f  \phi^{+}_{h, N}    e^{-\tfrac{V_N}{hN}}  dx \Big)^2 
     -     
      \Big( \int_{\mb R^N}  f  \phi^{-}_{h, N}    e^{-\tfrac{V_N}{hN}}  dx  
       \Big)^2     .     
       \end{equation}
\end{proposition}
\noindent
The proofs of Proposition~\ref{MainPropAway} 
and Proposition~\ref{MainPropDiag} will be given respectively in Section~\ref{SectionAway} and Section~\ref{SectionDiag}.
\begin{proof}[Proof of Theorem~\ref{thMAIN}]
Thanks to Proposition~\ref{MainPropDiag} we can fix 
 $R_0, h'_0, C'>0$  
and, for every 
$h>0, N\in \mb N$, functions $\phi^{+}_{h, N} ,\phi^{-}_{h, N}  \in C_b^\infty(\mb R^N)$ such that~\eqref{QLBdiag} holds true
for all $h\in (0, h'_0]$, $N\in \mb N$, 
  $f\in C_b^\infty(\mb R^N)$ with    
  $\supp f  \subset  \Big\{x\in \mb R^N:     \tfrac 1N \sum_{k} (x_k - \overline x)^2 \leq  R_0^2   \Big\} $ and with $C'$ instead of $C$. 
  In particular, denoting 
  by $S_N$ a generic $3$-dimensional linear subspace of $C_b^\infty(\mb R^N)$ and picking for every $N\in \mb N, h>0$  a function $f^*_{h,N}\in S_N$ which 
  in $L^2(e^{-V/hN}dx)$ has norm one and is orthogonal
  to both  $\theta_{N,R_0} \phi^{+}_{h, N} $ and 
  $\theta_{N, R_0} \phi^{-}_{h, N} $, one obtains 
 \begin{equation}\label{contribIIMS} \mc E_{h,N} [ \theta_{N, R_0}f^*_{h, N}]     \geq      C'    \int_{\mb R^N}  
  |\theta_{N, R_0}f^*_{h, N}|^2 e^{-\tfrac{V_N}{hN}}  dx  
      \   \ \    \forall h\in (0, h'_0], N\in \mb N
       .     
       \end{equation}
Moreover it follows from Proposition~\eqref{MainPropAway}
that there exist $C'', h_0''>0$ such 
that for all $h\in(0,  h''_0]$, $N\in \mb N$ 
 it holds
\begin{equation}\label{contribIIIMS} \mc E_{h,N} [\tilde \theta_{N, R_0} f^*_{h, N} ]     \geq      h^{-1} C''    \int_{\mb R^N}  |\tilde \theta_{N, R_0} f^*_{h, N}|^2 e^{-\tfrac{V_N}{hN}}  dx   . 
  \end{equation}
Recalling the definition~\eqref{definitionkappa} and putting together the decomposition~\eqref{firstIMSdecomp}, 
the estimates~\eqref{contribIIMS},~\eqref{contribIIIMS}
  and the bound~\eqref{Boundonlocalizationerror} on the localization error
  one gets for every $N\in \mb N$ and $h\in (0, \min\{1, h_0',h_0''\})$
  the lower bound  
  \begin{gather*}
  \kappa^{(S_N)}_{h,N}   \geq  
    C'    \int_{\mb R^N}  
  |\theta_{N, R_0}f^*_{h, N}|^2 e^{-\tfrac{V_N}{hN}}  dx    + 
    C''    \int_{\mb R^N}  |\tilde \theta_{N, R_0} f^*_{h, N}|^2 e^{-\tfrac{V_N}{hN}}  dx    -     h c(R_0)     \geq     \\
   \geq     \      \min\{C', C''\}  -     h c(R_0) . 
      \end{gather*}
Taking $C:=  \tfrac 12 \min\{C', C''\} $, $h_0 := \min\{1, h_0',h_0'', \tfrac{c(R_0)}{C}\}$ and recalling Definition~\eqref{minmaxdef}
finishes the proof.        
  \end{proof}

\section{Auxiliary Tools}  \label{sec.Aux}  

\noindent
In this section we review briefly some auxiliary tools which will be used in the remainder of the paper. 
The key ingredient in the proofs of  Proposition~\ref{MainPropAway}
and Proposition~\ref{MainPropDiag} 
is the so-called NGS Bound~\cite{Gross, Si_EQFT}, which we recall below for the sake of the reader. We shall use the following standard conventions. 
We say that a probability measure $m$ on $\mb R^d$ satisfies a 
logarithmic Sobolev inequality with constant $\rho>0$ if for all $f\in C_b^\infty(\mb R^d)$ the inequality 
\begin{equation}
\int_{\mb R^d}     |\nabla f|^2 dm    
 \geq    \frac \rho 2 \,  {\rm{Ent}}_{m}[f^2]
\end{equation}
holds true, where  we have denoted
 by 
 \[{\rm{Ent}}_{m}[f^2] : = 
\int_{\mb R^d}   f^2 \log f^2 dm   -  \int_{\mb R^d}   f^2  dm  \log  \int_{\mb R^d} f^2 dm  \]
 the entropy of $f^2$ with respect to $m$.
\begin{proposition}[NGS Bound] \label{NGSbound}
Let $m$ be a probability measure  on $\mb R^d$
and let 
$M:\mathbb R^d\ra\mathbb R$ and $\Omega \subset \mb R^d$ 
be respectively a continuous function and an open set. 
If there exist constants $\rho, \Lambda>0$ such that
$m$ satisfies a logarithmic Sobolev inequality with constant $\rho$
and such that 
$\int_{\Omega} e^{-  \frac{2M}{\rho}} dm \in (0,  \Lambda]$, 
then  
 \begin{equation} \label{NGSboundestimate}   
   \int_{\mb R^d}   |\nabla f |^2 dm       +    
 \int_{\mb R^d}   M |f |^2 dm          \geq       C    \int_{\mb R^d}   |f |^2 dm   \    \     ,
\   \    \forall f \in C_b^\infty(\mb R^d) : \supp f \subset \Omega  ,   \end{equation}
 where the contant $C\in \mb R$ is given by 
 \[ C  \   :=   \       - \tfrac \rho 2 \log  \Lambda   .    \]
\end{proposition}
\noindent
Note that in case that $M$ is bounded from below by a constant $M_0$
one recovers the trivial bound $C= M_0$. But crucially, 
for the finiteness of the integral $\int_{\Omega} e^{-  \frac{2M}{\rho}} dm$,
the boundedness of $M$ from below is not needed. 
\begin{remark}
The NGS Bound is usually stated with $\Omega=\mb R^d$. The slightly more general variant presented in Proposition~\ref{NGSbound} can be easily proven by repeating the original proof given in~\cite{Gross} and inserting suitable indicator functions of $\Omega$. 
\end{remark}

\noindent
The NGS Bound will be used in combination with the following two facts.
The first one follows from a simple computation which permits to transform by a unitary transformation the quadratic form $\mc E_{h.N}$
into an equivalent quadratic form where a $0$-order term appears 
as in the left hand side of~\eqref{NGSboundestimate}. This transformation is also called ground state transformation. More precisely the following holds. 
\begin{lemma}[Ground State Transformation]\label{GST}
Let $U,  W \in C^2(\mb R^d)$. Then for every $f\in C^\infty_c(\mb R^d)$, defining
 $ g :=  \exp{\left( -\tfrac{U-W}{2} \right)}$, one has 
\[    \int_{\mb R^d} |\nabla f |^2 e^{-U} dx      =  \int_{\mb R^d} |\nabla g |^2 e^{- W} dx
+    \int_{\mb R^d}   M   g^2 e^{- W} dx       ,            \] 
where $M:\mb R^d  \to \mb R$ is given by
\[    M   :  =   \tfrac 14 \left( |\nabla U|^2         -   |\nabla W|^2   \right)
 - \tfrac 12 \Delta(U-W)    .    \]
\end{lemma}
\noindent
The second important ingredient in order to exploit the NGS-Bound is the well-known Bakry-\'Emery criterion~\cite{BE}. The latter permits to 
give a quantitative bound on the logarithmic Sobolev constant in case of a uniformly convex potential (see also~\cite{BaBodineau} for an extension to singular nonconvex potentials).  
For the sake of the reader we recall the precise statement  
of the Bakry-\'Emery criterion, which we use in the present paper.
\begin{proposition}[Bakry-\'Emery criterion]\label{BEcriterion}
Let $U\in C^2(\mb R^d)$ and assume that there exists a $C>0$  
such that $\hess U(x) \geq C$ for all $x\in \mb R^d$. Then the probability measure $m(dx)$ on $\mb R^d$ proportional to $e^{-U}dx$ satisfies 
a logarithmic Sobolev inequality with constant $C$. 
\end{proposition}
\noindent
We shall use also the following fact, which might be deduced
as a simple corollary of the Bakry-\'Emery criterion:
if $U\in C^2(\mb R^d)$ and if there exists a $C>0$  
such that $\hess U(x) \geq C$ for all $x\in \mb R^d$, then the probability measure $m(dx)$ on $\mb R^d$ proportional to $e^{-U}dx$ satisfies 
 the 
following Poincar\'e inequality: 
 \begin{equation} \label{PoincBE}   
   \int_{\mb R^d}   |\nabla f |^2 dm          \geq       C  \left(   \int_{\mb R^d}   f^2 dm
   -       \big(\int_{\mb R^d} f dm     \big)^2 \right)    \    \     ,
\   \    \forall f \in C_b^\infty(\mb R^d)   .   
\end{equation}

\section{Proofs of the local estimates} \label{sec.Loc}   
\noindent
This section is devoted to the proofs of Proposition~\ref{MainPropAway} and Proposition~\ref{MainPropDiag}. We first introduce some notation and discuss basic porperties of the interaction part in the energy $V_N$. 

\

\noindent
Following the notation of \cite{DGLP} we denote by 
$K= K_N:\mb R^N \ra \mb R^N$ the normalised discrete Laplacian defined by setting for
 $x\in \mb R^{N}$ and $k\in\{1, \dots, N\}$
  \begin{equation*}\label{defdiscreteLaplace}   (K  x)_k   :=   \frac{\mu}{4\sin^2(\frac \pi N)} \big(  2 x_k  -    x_{k+1}  -   x_{k-1}    
   \big)     ,  
 \end{equation*}
 with the conventions $x_{N+1}:=x_1$ and $x_{0}:=x_N$. The interaction term in the energy $V_N$ can then  be written more compactly in terms of $K$, since 
 \[  \frac{\mu}{8\sin^2(\frac{\pi}{N})}   \   \sum_{k=1}^N  
   (x_{k} - x_{k+1})^2       =     \tfrac 12  \lp x, Kx \rp     ,   \]
   where $\lp \cdot, \cdot\rp$  is the standard scalar product in $\mb R^N$.
 The operator $K$ is diagonalised
through the discrete Fourier transform 
$\hat x\in \mb R^N$ of $x\in \mb R^N$, defined by
\[    \hat x_k    \  :  =    \    \frac{1}{\sqrt N}   \sum_{j=1}^{N}    x_j \    e^{-\im 2\pi \frac{j}{N}k}     \ .     \]
More precisely we have for every $k\in \{0, \dots, N-1\}$, 
\begin{equation}\label{eigenvaluesK}   (\widehat{Kx})_k    =     \nu_k \  \hat x_k    \   ,   \    \     \   \text{ where }    \   
\nu_{k}   =  \nu_{k, N}   :=       \mu   \,  \frac{\sin^2(k\frac \pi N)}{\sin^2(\frac \pi N)}        . 
\end{equation}
Note that  $\nu_k \sim k^2$ for large $N$. As a consequence, 
 due to the convergence of the series $\sum_{k=1}^\infty \tfrac{1}{k^2}$,
the resolvent of $K$ is uniformly trace class in $N$.
More precisely we shall use the following fact, whose proof is elementary.  
\begin{lemma}\label{remarkseries}
For every $\alpha> -\mu$ there exists a $\gamma(\alpha)>0$ such that for every $N\in \mb N$ 
we have
$  \sum_{k=1}^{N-1} \tfrac{1}{\nu_k + \alpha}   \leq    \gamma(\alpha) $.
\end{lemma}
\noindent
Observe also that $\nu_0=0$ is a simple eigenvalue of $K$ corresponding to the eigenspace of constant states and that its smallest non-zero eigenvalue equals $\mu$ for every 
$N\in \mb N, N\geq 2$. This implies immediately the following  Poincar\'e inequality, where we recall that $\overline x =   \tfrac 1N   \sum_k x_k =   \tfrac{1}{\sqrt N} \hat x_0$:
\begin{equation}\label{discretePoincare}
\lp x, K x \rp     \geq     \mu \sum_{k}    (x_k   -   \overline x)^2 
   \ \  \     ,  \  \   \    \forall x\in \mb R^N . 
\end{equation}
We note also the following Sobolev-type inequality, which, 
at the cost of lowering the constant $\mu$, allows to
substitue on the right hand side of~\eqref{discretePoincare}
the Euclidean norm with a supremum norm. This will be useful later in analysing the regions of convexity of $V_N$. 
\begin{lemma}\label{Sobolevlemma}
Let  $\gamma(0)>0$ as in Lemma~\ref{remarkseries}. Then 
\[    \lp x, K x \rp    \geq  \tfrac{1}{\gamma(0)}   N\sup_k |x_k - \overline x|^2       
   \ \   \     ,    \ \  \    \forall   x\in \mb R^N   . \] 
\end{lemma}
\begin{proof}
Let $x\in \mb R^N$ and assume that $\overline x =0$. Then 
\begin{gather*}
\sup_k |x_k|      \leq    \tfrac{1}{\sqrt N} \sum_{k=0}^{N-1}
|\hat x_k| =   
\tfrac{1}{\sqrt N} \sum_{k=1}^{N-1} \sqrt{\nu_k}
|\hat x_k|   \tfrac{1}{ \sqrt{\nu_k}}   \  \leq     \\   
\leq    \ 
\tfrac{1}{\sqrt N}   \left( \sum_{k=1}^{N-1} \nu_k
|\hat x_k|^2  \right)^{\tfrac 12}    \sqrt{\gamma(0) }     =  
\tfrac{1}{\sqrt N}    \lp x, Kx \rp^{\tfrac 12}    \sqrt{\gamma(0) } ,     
\end{gather*}
which finishes the proof in this case. 
If $\overline x \neq 0$ one can apply the argument to $x-\overline x$. 
\end{proof}

\subsection{Away from the diagonal: Proof of Proposition~\ref{MainPropAway}}\label{SectionAway}

\noindent
In order to prove Proposition~\ref{MainPropAway} we consider
the decomposition $\mb R^N = \mc C \oplus \mc C^{\perp}$, where
$\mc C:=  \{ x\in \mb R^N: x- \overline x =0\}$ (the space of constant states) and 
$\mc C^{\perp}:=  \{ x\in \mb R^N: \overline x =0\}$ (the space of 
mean zero states). On $\mc C$ we consider the coordinate $\xi$
with respect to the (non-normalized) basis vector $(1, \dots, 1)$. 
On $\mc C^{\perp}$ we fix an arbitrary orthonormal basis, denote by 
$y = (y_1, \dots, y_{N-1})$ the corresponding coordinates and by 
$A$ the $N\times N-1$ matrix so that each $x\in \mb R^N$
is uniquely determined by its coordinates $(\xi, y)\in (\mb R, \mb R^{N-1})$ via 
\[     x =   \xi  (1, \dots, 1)   + A y  .   \]
Note that in the $(\xi, y)$ coordinates we have
\[  \tilde V_N(\xi, y) :=  V_N( \xi  (1, \dots, 1)   + A y)   = 
F^{(\xi)}_N (y)    +  \tfrac N4 (\xi^2-1)^2    ,   \]
where for each $\xi \in \mb R$ and $N\in \mb N$ we 
have defined the function $F^{(\xi)}_N: \mb R^{N-1} \to \mb R$ by 
\[   F^{(\xi)}_N (y) :  =    \tfrac 14 P_4 (y) + \xi P_3(y) +    \tfrac 32 \xi^2  P_2(y)   +   \tfrac 12 Q(y)   , \]
with
\[Q(y) :=     \tfrac{\mu}{4\sin^2(\frac{\pi}{N})}   \   \sum_{k=1}^N  
   \left( (Ay)_{k} - (Ay)_{k+1} \right)^2    -    P_2(y) ,  \]
and with  $P_m(y): = \sum_{k=1}^N (Ay)_k^m$ for $m=2,3,4$. Note that $P_2(y) = \sum_k y_k^2$ for every $y\in \mb R^{N-1}$, since
we have chosen the coordinates $y$ to be orthonormal.

\begin{proposition} \label{propNGSawayfromdiagonal}
Fix $R>0$. Then for every $\xi\in \mb R$, $N\in \mb N$ and for every 
$\phi\in C_b^\infty(\mb R^{N-1}) $ such that 
$\supp \phi \subset \{y \in \mb R^{N-1}: P_2(y) \geq NR^2\}$
we have 
\[  hN  \int_{\mb R^{N-1}}  |\nabla \phi(y)|^2 
e^{-\tfrac{F^{(\xi)}_N(y)}{hN}} dy     \geq  
  C(h, N, \xi)     \int_{\mb R^{N-1}}  | \phi(y)|^2 
e^{-\tfrac{F^{(\xi)}_N(y)}{hN}} dy        ,  \]
  where 
  \[  C(h, N, \xi)   :=  \frac{\mu-1}{2}   \left( 
   \log  \int_{\mb R^{N-1}}  e^{-\tfrac{F_N^{(\xi)}(y)}{hN}}    dy 
-   \log \int_{\{P_2 \geq NR^2\}}
e^{-\tfrac{F_N^{(\xi)}(y)}{hN}}     dy \right)   .   \]
\end{proposition}

\begin{proof}
It follows from the Poincar\'e inequality~\eqref{discretePoincare}  
 that the Hessian of $y\mapsto F_N^{(\xi)} (y)$ is strictly positive, uniformly in $y$ and $\xi$. More precisely, 
\[     \hess F_N^{(\xi)} (y)    \geq    \mu -1        \  \  \         \forall \xi \in \mb R, y\in \mb R^{N-1}   .  \]
It follows then from the Bakry-\'Emery criterion (Proposition~\ref{BEcriterion}) that for each $\xi\in \mb R$ the probability measure
on $\mb R^{N-1}$  with density proportional to $\exp{(- \tfrac{ F_N^{(\xi)} }{h N})}$ satisfies a logarithmic Sobolev inequality with constant $\tfrac{\mu-1}{hN}$. Thus the claim follows by applying
for each $\xi\in \mb R$ to this probability measure
the NGS Bound (Proposition~\ref{NGSbound}) with $M=0$ and 
$\Omega = \{ P_2 \geq NR^2\}$. 
\end{proof}

\noindent
In order to estimate the constant $C(h, N, \xi)$ defined in Prop.~\ref{propNGSawayfromdiagonal} we shall use the following two simple lemmata, which give uniform estimates on suitable Gaussian integrals. 
We define for $t> -(\mu -1)$, $h>0$ and $N\in \mb N$
\[    Z_{h, N}(t)  :=      \int_{\mb R^{N-1}} e^{- \tfrac{t P_2(y) + Q(y)}{2 hN}} dy   =    h^{\tfrac{N-1}{2}}Z_{1, N}(t)     , \]
and write for short $Z_{N}(t)  := Z_{1, N}(t) $   . 

\begin{lemma} \label{Lemma4moment}
Fix $t_0> -(\mu -1)$. Then there exists a constant $C= C(t_0)>0$ such that for all $t\geq t_0$ and for all $N\in \mb N$ it holds 
\[\tfrac{1}{Z_{N} (t)}  \int_{\mb R^{N-1}}   \tfrac 1 N P_4 (y) e^{-\tfrac{
 t P_2(y)   +    Q(y)}{2N}}     dy   \leq   C.        \]
\end{lemma}

\begin{proof}
Fix $t_0> -(\mu-1)$ and let 
$\nu_0, \dots, \nu_{N-1}$ be the eigenvalues of $K$ given by~\eqref{eigenvaluesK}. 
We denote for short 
for each $t\geq t_0$ by $g_t(dy)$ the Gaussian measure 
on $\mb R^{N-1}$ given by
$\tfrac{1}{Z_{N} (t)}   e^{-\tfrac{
 t P_2(y)   +    Q(y)}{2N}}     dy $. 
 Then for each $t\geq t_0$ and $N\in \mb N$
 an explicit computation of Gaussian 4th order moments gives
  \begin{gather*}
  3\left( 1 + \sum_{k=1}^{N-1}   \frac{1}{\nu_k -1 +t }\right)^2 = 
  \tfrac{1}{Z_{N} (t) \sqrt{2\pi N}}  \int_{\mb R^N}
  \tfrac 1N \sum_k x_k^4 \, e^{- \tfrac{\lp (K-1 +t)x,x \rp}{2N} - \tfrac{(2-t)\overline x^2}{2}}  dx  = \\ 
   \int_{\mb R} \left(   \int_{\mb R^{N-1}}  
  \tfrac 1N \sum_k (\xi + y_k)_k^4   \, g_t(dy) \right) 
   \tfrac{1}{ \sqrt{2\pi }} e^{- \tfrac{\xi^2}{2}} 
   d\xi \,   
  =   \\ 
    \tfrac{1}{ \sqrt{2\pi}}  \int_{\mb R}   \xi^4 e^{- \tfrac{\xi^2}{2}} \, d\xi 
    +    \tfrac{6}{Z_{N} (t)}  \int_{\mb R^{N-1}}   \tfrac 1 N P_2 (y) \, g_t(dy)    
       \  \tfrac{1}{ \sqrt{2\pi}}  \int_{\mb R}   \xi^2 e^{- \tfrac{\xi^2}{2}} \, d\xi  
    +      \int_{\mb R^{N-1}}   \tfrac 1 N P_4 (y) \, g_t(dy)          =   \\  
     3   + 6  \int_{\mb R^{N-1}}   \tfrac 1 N P_2 (y) \, g_t(dy)    
    +    \int_{\mb R^{N-1}}   \tfrac 1 N P_4 (y) \, g_t(dy)        .  
  \end{gather*}
  Since
  \[   \int_{\mb R^{N-1}}   \tfrac 1 N P_2 (y) \, g_t(dy)   =  
  \sum_{k=1}^{N-1}   \frac{1}{\nu_k -1 +t } ,  \]
 we obtain
\begin{gather*}  \int_{\mb R^{N-1}}   \tfrac 1 N P_4 (y) \, g_t(dy)     = 
3\left( 1 + \sum_{k=1}^{N-1}   \frac{1}{\nu_k -1 +t }\right)^2  - 3 - 6  \sum_{k=1}^{N-1}   \frac{1}{\nu_k -1 +t }    \leq   \\
 3\left( 1 + \sum_{k=1}^{N-1}   \frac{1}{\nu_k -1 +t_0 }\right)^2,
\end{gather*} 
 which finishes the proof by Lemma~\ref{remarkseries}.    
\end{proof}
\begin{lemma} \label{lemmaRatioNorm}  For every $t_0 > - (\mu -1)$ there exists a $\gamma(t_0) >0$ 
such that for all $N\in \mb N$ and for all $t  >  - (\mu -1) $ it holds
\[       \frac{Z_N(t)}{Z_N(t_0)}     \geq      e^{  -    \tfrac{\gamma (t_0)|t- t_0| }{ 2}  }       .      \] 
\end{lemma}
\begin{proof} 
Fix $t_0>  - (\mu -1)$ and 
let 
$\nu_0, \dots, \nu_{N-1}$ be the eigenvalues of $K$ given by~\eqref{eigenvaluesK}. 
 Then, using $\log(1+x) \leq x$ for $x> -1$, 
one gets for all $N\in \mb N$ and for all $t  >  - (\mu -1) $ 
\begin{gather*}
 2 \log \left(\frac{Z_N(t)}{Z_N(t_0)} \right) =  
   - \sum_{k=1}^{N-1} \log  \left(\frac{\nu_k -1 + t}{\nu_k -1 + t_0}\right)
  =   \\
  -  \sum_{k=1}^{N-1} \log  \left(
  1     +    \frac{t - t_0}{\nu_k -1 + t_0}\right)     \geq 
   -  (t- t_0)\sum_{k=1}^{N-1}
   \frac{1}{\nu_k -1 + t_0}     \geq - |t-t_0| \gamma(t_0) ,   
\end{gather*}
where $\gamma(t_0)$ satisfies 
$\sum_{k=1}^{N-1}
   \frac{1}{\nu_k -1 + t_0}  \leq \gamma(t_0)$
   for all $N\in \mb N$ (see Lemma~\ref{remarkseries}). 
\end{proof}
\noindent
The next proposition provides an estimate on the constant $C(h, N, \xi)$ defined in Prop.~\ref{propNGSawayfromdiagonal}. 
\begin{proposition} \label{propcompofNGSconstant}
Fix $R>0$. Then there exist constants $C, h_0>0$ such 
that for every $h\in(0, h_0]$,  $N\in \mb N$ and $\xi\in \mb R$
it holds 
\[   \log  \int_{\mb R^{N-1}}  e^{-\tfrac{F_N^{(\xi)}(y)}{hN}}    dy 
-   \log \int_{\{P_2 \geq NR^2\}}
e^{-\tfrac{F_N^{(\xi)}(y)}{hN}}     dy         \geq   \tfrac C h   .  \]
\end{proposition}

\begin{proof} Fix $R>0$. We first show that 
there exists   a constant  $C'>0$ such that for all 
$h>0,  N\in \mb N$ and $\xi \in \mb R$ it holds 
\begin{equation}\label{laplaceintegralforf}    \int_{\mb R^{N-1}} e^{-\tfrac{F_N^{(\xi)}(y)}{hN}}     dy     \geq  
  \left(    1   -     h   C'   \right) Z_{h, N} (3\xi^2).    
  \end{equation}
 In order to do so we write for short 
\[  \tilde Z_{h, N}(3\xi^2)   :=
\int_{\mb R^{N-1}} e^{- \tfrac{F^{(\xi)}_N(y) - \xi P_3(y)}{ hN}} dy  ,     \]
and   $\mb P^{(\xi)}_{h, N}$ for the probability measure on $\mb R^{N-1}$ with density 
$ e^{- \tfrac{F^{(\xi)}_N(y) - \xi P_3(y)}{ hN}} / \tilde Z_{h, N}(3\xi^2) $.
It follows then from Jensen's inequality, the symmetry of 
$\mb P^{(\xi)}_{h, N}$ and the antisymmetry of $P_3$ that 
\begin{gather*}
\int_{\mb R^{N-1}} e^{-\tfrac{F^{(\xi)}_N(y)}{hN}}     dy      =   
\tilde Z_{h, N}(3\xi^2)   \int_{\mb R^{N-1}} 
e^{-\tfrac{\xi P_3 (y)}{hN}}     \mb P_{h, N}(dy)  \geq   \\
   \tilde Z_{h, N}(3\xi^2)   \ 
e^{-   \tfrac{\xi}{hN}  \int_{\mb R^{N-1}}  P_3 (y)   \mb P_{h, N}(dy)} 
    =       \tilde Z_{h, N}(3\xi^2)   .  
\end{gather*}
Moreover, using for the quartic term the inequality $e^{-t}\geq 1-t$, $t\in \mb R$, one gets
\begin{gather*}   \tilde Z_{h, N}(3\xi^2)    =   
\int_{\mb R^{N-1}}    e^{- \tfrac{ P_4(y)}{4 hN}}
e^{-\tfrac{
 3\xi^2 P_2(y)   +    Q(y)}{2h N}}      dy     \geq    \\
 \int_{\mb R^{N-1}}   
e^{-\tfrac{
 3\xi^2 P_2(y)   +    Q(y)}{2h N}}      dy   
 -    \tfrac{1}{4 hN } \int_{\mb R^{N-1}}    P_4(y) 
e^{-\tfrac{
 3\xi^2 P_2(y)   +    Q(y)}{2h N}}      dy       =   \\
 Z_{h,N}(3\xi^2)   \left(    1  -    
  h\tfrac{1}{4 N } \int_{\mb R^{N-1}}    P_4(y) 
e^{-\tfrac{
 3\xi^2 P_2(y)   +    Q(y)}{2 N}}   Z^{-1}_{1,N}(3\xi^2)         dy        \right) 
   .   
 \end{gather*}
The claim~\eqref{laplaceintegralforf} follows then by applying Lemma~\ref{Lemma4moment}  with $t_0=0$ and $t= 3\xi^2$.

 \

 \noindent
 As second step  we pick a $\delta\in (0, \mu-1)$ and show that
 for all $\xi\in \mb R$, $h>0$
and $N\in \mb N$ the estimate
\begin{equation}\label{TailestimateF}
\int_{\{P_2 \geq NR^2\}}
e^{-\tfrac{F_N^{(\xi)}(y)}{hN}}     dy     \leq   
 e^{- \tfrac{\xi^2+ \delta}{2h} R^2}   Z_{h, N}(-\delta)      
 \end{equation}
 holds true. 
Indeed, using with $\alpha=2$ the inequality
 $\xi t^3 \leq \tfrac{\alpha \xi^2 t}{2} + \tfrac{t^4}{2\alpha}$, 
 valid for every $\xi, t\in \mb R$ and $\alpha>0$, one gets 
 for every $N\in \mb N$, $\xi \in \mb R$ and $y\in \mb R^{N-1}$
 the lower bound
 \[   F^{(\xi)}_N (y)    \geq \tfrac 12 \xi^2 P_2(y) + \tfrac 12 Q(y)   . \]
 It follows that for all $\xi\in \mb R$, $h>0$
and $N\in \mb N$ it holds
\begin{gather*}\int_{\{P_2 \geq NR^2\}}
e^{-\tfrac{F_N^{(\xi)}(y)}{hN}}     dy     \leq  
\int_{\{P_2 \geq NR^2\}} 
e^{-\tfrac{ (\xi^2 + \delta) P_2(y) }{2hN}}  e^{-\tfrac{ Q(y) - \delta P_2(y)}{2hN}}       dy  \leq  \\ 
 e^{- \tfrac{\xi^2+ \delta}{2h} R^2}   Z_{h, N}(-\delta)   ,   
 \end{gather*}
 i.e. \eqref{TailestimateF} is proven. 

 \
 
 \noindent
 To finish the proof of the proposition we observe that 
 it follows  from~\eqref{laplaceintegralforf}, ~\eqref{TailestimateF} and Lemma~\ref{lemmaRatioNorm} that there exists 
constants $ C', \delta'>0$ such that, chosing $ h'_0 \in (0,
\tfrac {1}{2 C'})$, 
 for every $h\in(0,  h'_0]$,  $N\in \mb N$ and $\xi\in \mb R$
 it holds
\begin{gather*}
 \log  \int_{\mb R^{N-1}}  e^{-\tfrac{F_N^{(\xi)}(y)}{hN}}    dy 
-   \log \int_{\{P_2(y) \geq NR^2\}}
e^{-\tfrac{F_N^{(\xi)}(y)}{hN}}     dy     \    \geq   \\
     \geq    \   \log \left(    1   -     h    C'   \right)  + 
   \tfrac{\xi^2+ \delta}{2h} R^2    -    \tfrac{\delta'}{2}\left( 3\xi^2 + \delta\right)  
    \   \geq     \\     \geq     \      
    \tfrac 12 \left(  \tfrac {R^2}{h} - 3\delta'\right)
   \xi^2      +    \tfrac 1 h \left( \tfrac {R^2\delta }{2} - 
   h \delta( \tfrac{ \delta'}{2} - \tfrac{\log 2}{\delta}) \right)
     . 
\end{gather*}
It follows that, fixing any $C <  \tfrac{R^2\delta}{2}$, we can find an
$h_0 \in (0, h_0']$ such that the statement of the proposition holds. 
\end{proof}

\noindent
We can now easily complete the proof of Proposition~\ref{MainPropAway}
by integrating over the diagonal: 
\begin{proof}[Proof of Proposition~\ref{MainPropAway}]
Fix $R>0$.
It follows from Proposition~\ref{propNGSawayfromdiagonal} and 
Proposition~\ref{propcompofNGSconstant} that
there exists  constants $C, h_0>0$ such that for every $N\in \mb N$, 
every $h\in (0, h_0]$ and every $f\in C_b(\mb R^N)$ with the property that
$ \supp f  \subset   \Big\{x\in \mb R^N:     \tfrac 1N \sum_{k} (x_k - \overline x)^2 \geq  R^2   \Big\} $, denoting 
for each $\xi\in \mb R$ by $f^{(\xi)}$ the function 
$y\mapsto  f(x(\xi, y))\in C_b^\infty(\mb R^{N-1)}$
and by $\nabla f^{(\xi)} $ its gradient, it holds  
\begin{gather*}
\mc E_{h,N} [f]     \geq  hN  \int_{\mb R}  \left(  \int_{\mb R^{N-1}}
|\nabla f^{(\xi)}(y)|^2 e^{- \tfrac{F^{(\xi)(y)}}{hN }} dy\right) e^{-\tfrac{1}{4h}(\xi^2-1)^2}  \sqrt N d\xi   \geq \\
  \tfrac C h \int_{\mb R}  \left(  \int_{\mb R^{N-1}}
| f^{(\xi)}(y)|^2 e^{- \tfrac{F^{(\xi)(y)}}{hN }} dy\right) e^{-\tfrac{1}{4h}(\xi^2-1)^2}  \sqrt N d\xi      =   
 \tfrac{C}{h}    \int_{\mb R^N}  f^2 e^{-\tfrac{V_N}{hN}}  dx      . 
\end{gather*}

\end{proof}

\noindent

\subsection{Around the diagonal: Proof of Proposition~\ref{MainPropDiag}}
\label{SectionDiag}

\noindent
In order to prove Proposition~\ref{MainPropDiag} we shall consider another quadratic partition of unity, which permits to 
isolate and treat separately various contributions to $\mc E_{h, N}$: 
the contribution 
coming from a large convexity region of $V_N$ (containing the two local minima $I_+, I_-$), see Proposition~\ref{MainPropDiagMinima} below; the contribution due to a neighbourhood of the saddle point, 
see Proposition~\ref{propsaddle} below; and finally the contribution coming from the small remaining region, see Proposition~\ref{propinflection} below. 

\

\noindent
We consider in the sequel for every $N\in \mb N$ and $R>0$
the strip around the diagonal given by 
\[    \mc S_N(R) :=     \Big\{x\in \mb R^N:     \tfrac 1N \sum_{k} (x_k - \overline x)^2 \leq  R^2   \Big\}  .\]
Moreover we define for all $N\in \mb N, R >0$ and 
$r  \geq 0 $ the sets
\[   \Omega^{\rm min}_N(R,r)    :=  \mc S_N(R)   \cap    
 \Big\{x\in \mb R^N:    |\overline x|    \geq r  \Big\}  ,      \]
 which, for $0\leq r<1$, are neighbourhoods of the two minima $\pm I$
 and the sets 
 \[
   \Omega^{0}_N(R,r)   :=    \mc S_N(R)   \cap  
   \Big\{x\in \mb R^N:   |\overline x | \leq r    \Big\}      ,  \]
 which, for $r>0$, are neighbourhoods of the saddle point $0$. 
   
   \
   
\noindent
As a preliminary step we show in Lemma~\ref{convexity} below that the energy $V_N$ is uniformly convex on suitable sectors containing the two global minima (and thus in particular on 
$\Omega^{\rm min}_N(R,r)$ if $R$ is small enough and $r$ is large enough).  This is a rather straightforward consequence
of the Sobolev inequality given in Lemma~\ref{Sobolevlemma}.   
   
  \begin{lemma}[Convexity around $I_+, I_-$]\label{convexity}
For every $r>\tfrac{1}{\sqrt 3}$ there exist constants $\alpha_0
= \alpha_0(r), C= C(r)>0$
such that for all $N\in \mb N$ and for all $x\in \mb R^N$ 
such that  $|\overline x| \geq r$ and 
$\sum_k (x_k - \overline x)^2 \leq N \alpha_0^2 |\overline x|^2$
one has the bound 
\begin{equation*}     \hess V_N (x)   \geq C     .  
\end{equation*}
In particular there exists an $R_0>0$ such that for every $r>\tfrac{1}{\sqrt 3}$ there exists a constant 
$C(r)>0$ such that 
\[     \hess V_N (x)   \geq C(r)          \  \  \   ,  \  \ \   \forall N\in \mb N
\text{ and } 
\forall x\in\Omega^{\rm min}_N(R_0,r)    .  \]
\end{lemma}

\begin{proof}
We note first that for all $N\in \mb N$, $x\in \mb R^N$ and $\omega\in \mb R^N$
such that $\sum_k \omega_k^2 =1$ it holds 
\[    \lp \hess V_N(x) \omega, \omega\rp  =    
3\sum_k x_k^2 \omega^2_k    -1    +      \lp K\omega, \omega\rp      .   \]
We fix $r> \tfrac {1}{\sqrt 3}$, $\ve>0$
 and let $N\in \mb N$ and $\omega\in \mb R^N$
with $\sum_k \omega_k^2 =1$. If 
$\lp K\omega, \omega\rp    \geq 1+\ve$ we use
$3\sum_k x_k^2 \omega^2_k\geq 0$ and  conclude that 
$\lp \hess V_N(x) \omega, \omega\rp      \geq \ve$
for all $x\in \mb R^N$. Thus we can assume in the sequel that 
$\lp K\omega, \omega\rp    \leq 1+\ve$. In this case, using
$\lp K\omega, \omega\rp \geq 0$ and the decomposition
$x =  x + (x-\overline x)$, one gets for every $x\in \mb R^N$ the estimate
\begin{gather*}
\lp \hess V_N(x) \omega, \omega\rp       
\geq 3\overline x^2 -1 + 6\overline x \sum_k(x_k - \overline x) \omega_k^2  + 3 \sum_k(x_k - \overline x)^2 \omega_k^2  \geq \\
\geq 3\overline x^2 -1 + 6\overline x \sum_k(x_k - \overline x) \omega_k^2    .  
\end{gather*}
It follows then from the decomposition $\omega =  \overline\omega + (\omega-\overline \omega)$ and the Cauchy-Schwarz inequality that, 
for every $\alpha>0$ and every $x\in \mb R^N$ with 
$\sum_k (x_k - \overline x)^2 \leq N \alpha^2 |\overline x|^2$, 
\begin{gather}\label{hessstep1}
\lp \hess V_N(x) \omega, \omega\rp       
\geq  3\overline x^2 -1 -12  \alpha \overline x^2 - 6\alpha \overline x^2
\left(N \sum_k (\omega_k - \overline \omega)^4\right)^{\tfrac 12}.  
\end{gather}
The Sobolev inequality of 
Lemma~\ref{Sobolevlemma} implies that there exists a constant $\gamma>0$ 
such that 
\begin{equation}\label{useSI}   \left(N \sum_k (\omega_k - \overline \omega)^4\right)^{\tfrac 12}
\leq  \gamma \lp K \omega, \omega \rp  \leq \gamma ( 1+\ve)   . 
\end{equation}
Thus~\eqref{hessstep1} implies that, 
for every $\alpha>0$ and every $x\in \mb R^N$ with 
$\sum_k (x_k - \overline x)^2 \leq N \alpha^2 \overline x^2$
and $|\overline x| > r$, defining for short $\delta:= 3r^2-1 >0$
and $t:= 6 (2    -  \gamma(1+\ve) ) \in \mb R$,
\begin{gather*}
\lp \hess V_N(x) \omega, \omega\rp       
\geq  3\left[1   -  6 \alpha(2    -  \gamma(1+\ve) )
\right] r^2 - 1    =   \\
\left[1   -   \alpha t\right] (1+\delta) - 1   \geq     
\delta (1- \alpha |t|)  -  \alpha |t|      .\end{gather*}
We conclude that, taking $\alpha_0> 0$ such that 
$\alpha_0|t| \leq \min\{\tfrac 12, \tfrac\delta 4\}$, 
the estimate \begin{gather*}
\lp \hess V_N(x) \omega, \omega\rp       
\geq   \tfrac \delta 4     \end{gather*}
holds for
every $x\in \mb R^N$ with 
$\sum_k (x_k - \overline x)^2 \leq N \alpha_0^2 \overline x^2$
and $|\overline x| > r$. 
\end{proof}

\begin{remark}[Convex modifications of $V_N$]\label{convexremark}
It follows from Lemma~\ref{convexity} that there exists $R_0>0$ such that for every $r>\tfrac{1}{\sqrt 3}$ there exist a constant 
$C(r)>0$ and, for every $N\in \mb N$,
functions  $ V^{+}_{N,r}, V^{-}_{N,r} \in C^\infty(\mb R^N)$
with the following property: 
\begin{itemize}
\item[(i)]   $ \hess V^{\pm}_{N,r} (x)   \geq C(r)  $ for all $ N\in \mb N$
and   $x\in \mb R^N $, 
\item[(ii)]  $V^{\pm}_{N,r} (x)   =     V_{N} (x) $     
for all $ N\in \mb N$ and   $x\in  \Omega^{\rm min}_N(R_0,r)
 \cap \{ \pm \overline x \geq 0\}$.
\end{itemize}
\end{remark}
\begin{remark}
Note that it would be enough to have a Sobolev inequality 
for the $4$-norm appearing on the left hand side of~\eqref{useSI}
instead of the stronger $\infty$-norm statement of Lemma~\ref{Sobolevlemma}. This remark permits to extend our arguments to space dimension two and three. 
\end{remark}

\noindent
In regions where $V_N$ is uniformly convex we can now estimate 
$\mc E_{h,N} $ from below by the Bakry-\'Emery criterion for the spectral gap given by~\eqref{PoincBE}:

\begin{proposition}[Estimates in convex regions around the minima]\label{MainPropDiagMinima}
There exists an $R_0 >0$ such that for every 
$r\in (\tfrac{1}{\sqrt 3},1)$ the following holds: 
there exist constants $C= C(r)>0$ and, for every 
$h>0, N\in \mb N$, there exist  functions $\phi^{+}_{h, N} =  
\phi^{+}_{h, N, r}  $, $\phi^{-}_{h, N} =  
\phi^{-}_{h, N, r}  \in C_b^\infty(\mb R^N)$
 such that for all $h>0$, $N\in \mb N$ and 
  $f\in C_b^\infty(\mb R^N)$ with    
  $\supp f  \subset \Omega^{\min}_{N} (R_0, r)$ it holds
\[ \mc E_{h,N} [f]     \geq      C    \int_{\mb R^N}  f^2 e^{-\tfrac{V_N}{hN}}  dx  
     -    \Big(\int_{\mb R^N}  f  \phi^{+}_{h, N}    e^{-\tfrac{V_N}{hN}}  dx \Big)^2 
     -     
      \Big( \int_{\mb R^N}  f  \phi^{-}_{h, N}    e^{-\tfrac{V_N}{hN}}  dx  
       \Big)^2     .     \]
\end{proposition}
\begin{proof}
Take $R_0>0$ as in Remark~\ref{convexremark}, let 
$r\in (\tfrac{1}{\sqrt 3},1)$ and consider the constant $C=C(r)$
and the functions 
$ V^{+}_{N} := V^{+}_{N,r}$, $V^{-}_{N} := V^{-}_{N,r}$ as in Remark~\ref{convexremark}.
Moreover define for shortness $Z^{\pm}_{h, N} :=    \int_{\mb R^N} \exp{(-\tfrac{V^{\pm}_N}{hN})}  dx $ and  
 take $\chi^{+}, \chi^{-}\in C_b^\infty(\mb R^N) $ such that $\chi^{\pm} \equiv 1$ on $\{\pm \overline x\geq \tfrac r2\}$ and 
$\chi^{\pm} \equiv 0$ on $\{\pm \overline x\leq 0\}$.
 Then for 
all $h>0$, $N\in \mb N$ and 
  $f\in C_b^\infty(\mb R^N)$ with    
  $\supp f  \subset \Omega^{\min}_{N} (R_0, r)$ it 
  follows from~\eqref{PoincBE} that
\begin{gather*}
\mc E_{h, N}[f]   =   hN\int_{\mb R^N}  |\nabla(\chi^+ f)|^2  e^{-\tfrac{V^+_N}{hN}}  dx   +  
hN\int_{\mb R^N}  |\nabla (\chi^-f)|^2  e^{-\tfrac{V^-_N}{hN}}  dx       \  \geq   \\
     \geq   \   C \left( \int_{\mb R^N} |\chi^+ f|^2   e^{-\tfrac{V^+_N}{hN}}  dx   -    \Big(  \tfrac{1}{\sqrt{Z^+_{h, N}}}\int \chi^+ f   e^{-\tfrac{V^+_N}{hN}}  dx \Big)^2  \right)     \ 
 +  \\
   +      \
   C  \left( \int_{\mb R^N} |\chi^- f|^2   e^{-\tfrac{V^-_N}{hN}}  dx   -    \Big(  \tfrac{1}{\sqrt{Z^-_{h, N}}}\int \chi^- f   e^{-\tfrac{V^-_N}{hN}}  dx \Big)^2  \right) 
     =   \\
   =  
    C  \int_{\mb R^N} f^2   e^{-\tfrac{V_N}{hN}}  dx   -    \Big( \int  f  \sqrt{\tfrac{C}{Z^+_{h, N}}}\chi^+    e^{-\tfrac{V_N}{hN}}  dx \Big)^2  -
 \Big( \int  f  \sqrt{\tfrac{C}{Z^-_{h, N}}}\chi^-
   e^{-\tfrac{V_N}{hN}}  dx \Big)^2   
     . 
\end{gather*}
Thus the functions 
$\phi^{\pm}_{h, N} :=  
\sqrt{\tfrac{C}{Z^\pm_{h, N}}}\chi^\pm$     satisfy the statment of the proposition. 
\end{proof}
\noindent
The next proposition gives a lower bound on $ \mc E_{h,N}$ on  $\Omega^{0}_N(R,r)$ for $R,r$ small enough. The proof
is elementary if one first bounds the norm of the full gradient of $f$ from below by the norm of the directional derivative in direction of the constants and then performs a one-dimensional version of the ground state transformation (Lemma~\ref{GST}).  
\begin{proposition}[Estimate around the saddle point]\label{propsaddle}
 Let $\delta, C >0$ with $C< \delta^2< \tfrac 12$
and define $r(\delta):=   \sqrt{\tfrac{(1-2\delta^2)}{ 3}}>0$ and 
   $R(\delta):= \sqrt{ \tfrac 23(\delta^2 - C)}>0$. Then the following holds: for every 
 $h>0$, $N\in \mb N$ and 
  $f\in C_b^\infty(\mb R^N)  $ with   
  $ \supp f  \subset \Omega^{0}_{N}(R(\delta), r(\delta))$
\[ \mc E_{h,N} [f]     \geq      C    \int_{\mb R^N}  f^2 e^{-\tfrac{V_N}{hN}}  dx      .    \]
 \end{proposition}
\begin{proof}  
Let $f\in C_b(\mb R^N)$. 
Setting $g :=  e^{\tfrac{ - V_N }{2hN}} f$ we get 
\begin{gather}
\mc E_{h,N} [f]     \geq   h N \int_{\mb R^N} \big|\nabla f(x) \cdot \tfrac{(1, \dots, 1)}{\sqrt N }\big|^2   e^{- \tfrac{V_N}{hN}} dx =  h \int_{\mb R^N} \big|\sum_k \partial_k f \big|^2   e^{- \tfrac{V_N}{hN}} dx    \   =    \nonumber \\
=    h \int_{\mb R^N}   \big\{  \big|\sum_k \partial_k g \big|^2  
+  \tfrac{1}{4h^2N^2}\big|\sum_k \partial_k V_N \big|^2  g^2      -    \tfrac {1}{2hN}  \sum_{k,j} \partial^2_{k,j} V_N   \, 
g^2  \big \}  dx     \    \geq     \nonumber   \\
-   \  \tfrac {1}{2N}     \int_{\mb R^N} \sum_{k,j} \partial^2_{k,j} V_N  \, 
f^2 e^{-\tfrac{V_N}{hN}}     dx  .    \label{saddleestimate}
\end{gather}
Moreover for every $x\in \mb R^N$
\begin{gather*}
 - \tfrac {1}{2N} \sum_{k,j} \partial^2_{k,j} V_N(x)   =  
  - \tfrac {1}{2N} \sum_k (3x_k^2 -1)    =    \tfrac {1}{2}  - 
   \tfrac {3}{2} \overline x^2    - 
   \tfrac {3}{2N} \sum_k  (x_k - \overline x) ^2    . 
\end{gather*}
Thus for every $x\in \Omega^0_N(R(\delta)r(\delta))$
\begin{gather*}
 - \tfrac {1}{2N} \sum_{k,j} \partial^2_{k,j} V_N(x)   \geq   
     \tfrac {1}{2}  - 
   \tfrac {3}{2} r(\delta)^2   - 
   \tfrac {3}{2} R(\delta)^2    \geq C   . 
\end{gather*}
This gives the desired resultn by taking $f$ with 
  $ \supp f  \subset \Omega^{0}_{N}(R(\delta), r(\delta))$
in\eqref{saddleestimate}.
\end{proof}

\noindent
Note that the domains covered by Proposition~\ref{MainPropDiagMinima}
and Proposition~\ref{propsaddle} do not match. 
In order to fill the gap we consider now 
domains in $\mb R^N$ of the form 
  \[
   \tilde\Omega_N(R,r_1, r_2)   :=    \mc S_N(R)   \cap  
   \Big\{x\in \mb R^N:   r_1 \leq  |\overline x | \leq r_2
       \Big\}      .  \]
 where  $r_1, r_2 \in (0,1)$, $r_1< r_2$ and $R>0$. 
 One expects the restriction of the Dirichlet form $\mc E$
 on functions supported in such a $\tilde \Omega_{N}(r_1, r_2, R)$ to be  bounded from below by a term of order $h^{-1}$, uniformly in $N$. The next proposition shows that this is the case, at least if $(r_1, r_2) \subset (0, \tfrac {1}{\sqrt 2})$ and $R$ is sufficiently small. This will be enough for our purposes, since 
 for our main proof we will need only to cover the case in which 
 $(r_1, r_2)$ is an arbitrarily small neighbourhood of the inflection point
 $\tfrac {1}{\sqrt 3}$ and $R$ is arbitrarily small. 
\begin{proposition}[Estimates around infliction points]\label{propinflection}
For every $\delta\in (0,\tfrac 13)$ there exist constants $C= C(\delta), R_0 = R_0(\delta), h_0 = h_0(\delta)>0$ such that for all $h\in (0, h_0]$, $N\in \mb N$
and  $f\in C_b^\infty(\mb R^N)$  with $    \supp f  \subset \tilde \Omega_{N} (R_0, \sqrt{\delta}, 
\tfrac{1}{\sqrt 2 }\sqrt{ 1- \delta}) $ we have 
\begin{equation}\label{DirichletinflectionProp}
 \mc E_{h,N} [f]     \geq      \tfrac{C}{h}    \int_{\mb R^N}  f^2 e^{-\tfrac{V_N}{hN}}  dx      .     
   \end{equation}
\end{proposition}

\begin{proof}
Fix $\delta\in (0, \tfrac 13)$.  We set for short 
$\kappa := \min\{\tfrac \delta2, \mu-1\}$ and  
consider 
for each 
$N\in \mb N$ the
function $U_{N}:\mb R^N \to \mb R$ given by 
\begin{gather*}    U_{N}(x)   :=    V_N(x)   + \tfrac{1+\kappa}{2}    
N \overline  x^2     - \tfrac N4   =   
\tfrac 14 \sum_k x_k^4   + Q_{N}(x)  ,  
\end{gather*}
with $Q_{N}(x) :=    \tfrac \kappa 2 N \overline x^2   + 
\tfrac 12  \lp x-\overline x, (K-1) x- \overline x\rp  $.
Observe that $\hess U_{N}(x)   \geq \kappa >0$ for every 
 $N\in \mb N$ and every $x\in \mb R^N$. Further, given 
$f\in C_b^\infty(\mb R^N)$, according to Remark~\ref{GST} we compute  
\begin{gather*}
 \mc E_{h,N} (f)      =  \\  \int_{\mb R^N}   
 \left[  hN | \nabla g|^2    +   \left(  \tfrac{1}{4hN} |\nabla V_N|^2 - 
  \tfrac{1}{4hN} |\nabla U_{N}|^2    - \tfrac 12 \Delta 
  (V_N-U_{N})        \right)   g^2      \right]     e^{- \tfrac{U_{N}}{hN}}  dx,
\end{gather*}
where $g = e^{\tfrac{ - V_N + U_{N}}{2hN}} f$. 
It follows then from the NGS Bound (Prop.~\ref{NGSbound}),
the Bakry-Emery criterion (Prop.~\ref{BEcriterion})
and the estimate $\Delta 
  (V_N-U_{N}) (x) = - (1+\kappa)\leq 0$
that 
for every $R>0$, $N\in \mb N$ and  $f\in C_b^\infty(\mb R^N)$ with
 $\supp f  \subset \Omega_{N,\delta,R}:=
 \tilde \Omega_N\big( R, \sqrt{\delta}, 
\tfrac{1}{\sqrt 2 }\sqrt{ 1- \delta}\big)  $
we have 
\begin{equation}\label{dirichletinflection}\mc E_{h,N} (f)   \geq    C (h, N, R)   \int_{\mb R^N} f^2 
 e^{- \tfrac{V_{N}}{hN}}  dx, 
\end{equation}
where 
\[  C(h, N, R)   :=   
  \frac{\kappa}{2}   \left( 
   \log  \int_{\mb R^{N}}  e^{-\tfrac{U_{N}(x)}{hN}}    dx 
-   \log \int_{\Omega_{N,\delta, R}}
 e^{-\tfrac{G_{N}(x)}{hN}}     dx \right)  ,     \]
with 
\begin{gather*}   G_{N}   :=      \tfrac{1}{2\kappa}
\left( |\nabla V_N |^2 - |\nabla U_{N}|^2  \right)  +     U_{N}     =   \\
 \tfrac{1}{2\kappa}
 |\nabla (V_N - U_{N}) |^2  +   \tfrac {1}{\kappa}   \nabla U_{N} \cdot \nabla (V_N - U_{N}) 
    +     U_{N}    
.    \end{gather*}
A straigthforward computation (see below for details) 
shows that there exists a constant $c\in \mb R$
such that for every $N\in \mb N$ and every $R>0$ we have
the lower bound
\begin{equation}\label{LBonF_N}  G_N(x) 
   \geq N \left(  \tfrac{\delta^2}{4 \kappa}   -  c R^2  \right)    + Q_N(x)        \  \  \     ,   \  \ \    \forall x\in \Omega_{N, \delta, R}  .    
\end{equation}
Thus, letting $\tilde C( R) :=  \tfrac{\delta^2}{ 4\kappa}   - c R^2$ and 
denoting for short by $\gamma_N(dx)$ the Gaussian probability measure on $\mb R^N$ with density $e^{- \frac{Q_N(x)}{N}}$, one gets
for all $h,R>0, N\in \mb N$
\begin{gather*}  C(h, N, R)   \geq    
  \frac{\kappa}{2}   \left( 
   \log  \int_{\mb R^{N}}  e^{-\tfrac{U_{N}(x)}{hN}}    dx 
-   \log \int_{\mb R^{N}}
 e^{-\tfrac{Q_{N}(x)}{hN}}     dx \right)  +     \tfrac{\kappa \tilde C( R)}{2 h}
 =      \\
 =    \frac{\kappa}{2}   
   \log  \int_{\mb R^{N}}  e^{-\frac{h \sum_k x_k^4 }{4N}}    \gamma_N(dx) 
   +     \tfrac{\kappa \tilde C(R)}{2 h}   \geq   
   - h  \tfrac{\kappa}{8}  \tfrac{1}{N} \sum_k   
   \int_{\mb R^{N}}  x_k^4     \      \gamma_N(dx) 
    +     \tfrac{\kappa \tilde C(R)}{2 h}  . 
 \end{gather*}
An explicit computation of the Gaussian integrals
$\int_{\mb R^{N}}  x_k^4     \      \gamma_N(dx) $ (see also~\cite[Lemma 2.2] {DGLP}
for details) shows that there exists a constant $M>0$ 
such that 
\[   \tfrac{1}{N} \sum_k   
   \int_{\mb R^{N}}  x_k^4     \      \gamma_N(dx)    <    M     \  \  \   
   ,  \  \  \   \forall N\in \mb N. 
   \]
 Thus, chosing $R_0>0$ small enough, such that 
 $\tilde C(R_0) =  \tfrac{\delta^2}{4 \kappa}   - c R_0^2>0$
 and $h_0  =   \sqrt{\tfrac{  2 \tilde C(R_0)}{M} }$, one gets
 for every $N\in \mb N$ and every $h\in (0, h_0]$
\begin{gather*}  C(h, N, R_0)   \geq    \tfrac{\kappa \tilde C(R_0)}{4 h }  .    
 \end{gather*}
 Together with~\eqref{dirichletinflection} this provides the final estimate~\eqref{DirichletinflectionProp} with 
 $C:= \tfrac{\kappa \tilde C(R_0)}{4}$. 
 
 \

\noindent
In order to complete the proof, we give now the details on how the estimate~\eqref{LBonF_N} can be proven. First, to obtain a more explicit formula for $G_{N} $, we observe that 
 $\left[ \nabla(V_N- U_{N})(x)\right]_j =  - (1+\kappa) \overline x$
for every $j=1, \dots, N$ and that
\begin{gather*}
\sum_j \left[ \nabla U_{N} (x)\right]_j   =   
\sum_{j} \left( x_j^3 - x_j + [Kx]_j\right) + (1+\kappa)N \overline x = \\
= \sum_{j}  x_j^3  + \kappa N \overline x   =   
N\overline x^3    + \kappa N \overline x   + 3 \overline x \sum_{j} (x_j - \overline x)^2   + \sum_j (x_j - \overline x)^3 . 
\end{gather*}
It follows that 
\begin{gather*}
 G_{N} (x)  \ =  \\  \tfrac{(1+\kappa)^2}{2\kappa} N\overline x^2  -  
\tfrac{1+\kappa}{\kappa} \Big( N\overline x^4 + \kappa N \overline x^2 + 
3 \overline x^2 \sum_{j} (x_j - \overline x)^2  \   +   \\
+   \ \overline x\sum_j (x_j - \overline x)^3   \Big) + \tfrac 14 \sum_j x_j^4   + Q_{N}(x)  \  =  \\
=  \ \tfrac{N}{2\kappa} \theta_{N}(\overline x) \,    \overline x^2 
+ H_{N} (x)  
 + Q_{N}(x) , 
 \end{gather*}
 with 
 \[    \theta_{N}(\xi)   :=  
 \left[ 1 -\kappa^2 -\tfrac 12 (4+3\kappa) \xi^2     \right] , \]
 and
 \[    H_{N}(x) := 
\tfrac 14 \sum_{j} (x_j-\overline x)^4  
- \tfrac 1 \kappa  \overline x  \sum_{j}
 (x_j -\overline x)^3    
- \tfrac 32 \tfrac{2+\kappa}{\kappa} \overline x^2 \sum_j (x_j - \overline x)^2 . 
 \]
 Note that for $\overline x^2 \leq \tfrac 12 (1- \delta)$ 
  one has
 $\theta_{N}(\overline x) \geq \delta - (\kappa^2 +\tfrac 34 \kappa) 
 \geq \tfrac \delta 2 $, where for the last inequality we have used  
 $\delta< \tfrac 13$ and $\kappa\leq \tfrac \delta 2$. 
 Moreover, the inequality 
 $ab \leq    \alpha  a^2   +  \tfrac{1}{4\alpha}b^2$ with 
 $a= (x_k - \overline x)^2$, $b= \overline x(x_k - \overline x)$ and $\alpha =  \tfrac \kappa 4$ yields
 \begin{gather*}  
 H_{N} (x)    \geq -  \tilde c \overline x^2 \sum_j (x_j - \overline x)^2
    \  \  \    ,      \text{ with }   \tilde c :=  \tfrac 32 \tfrac{2+\kappa}{\kappa} 
 +  \tfrac{1}{\kappa^2}     .
\end{gather*}
It follows that
\begin{gather*}
 G_{N} (x) \geq 
 \left( \tfrac{\delta^2}{4\kappa}  -  \tfrac {\tilde c}{2}( 1 - \delta) R^2  \right)    N    
 + Q_{N}(x)    
   \  \  \    ,    \  \ \    \forall x\in \Omega_{N, \delta, R} ,
\end{gather*}
i.e.~\eqref{LBonF_N} with $c= \tfrac {\tilde c}{2}( 1 - \delta) $.

\end{proof}

\noindent
We can now wrap up all the estimates of this section and
prove Proposition~\ref{MainPropDiag}. 
\begin{proof}[Proof of Proposition~\ref{MainPropDiag}]
We fix $a_0,b_0, a_{{\rm min }}, b_{{\rm min }}>0$ such 
that 
\[a_0 < b_0< \tfrac{1}{\sqrt 3}<a_{{\rm min }}<b_{{\rm min }}< \tfrac{1}{\sqrt 2} \]
and  two functions $ \chi^{{\rm{min}}}, \chi_0
\in C_b^\infty(\mb R;[0,1])$ such that
$\chi^{{\rm{min}}} \equiv 1$ on $(-\infty, -b_{\rm{min}}]\cup [b_{\rm{min}}, +\infty)$, $\chi^{{\rm{min}}} \equiv 0$ on $[-a_{\rm{min}}, a_{\rm{min}}]$, 
$\chi^0 \equiv 1$ on
 $[-a_0, a_0]$, $\chi^0 \equiv 0$ on $\mb R\setminus [-b_0, b_0]$. We then define $\zeta_{N}^{{\rm{min}}}$, 
$ \zeta^{0}_{N} $, $\tilde \zeta_{N}: \mb R^N \ra [0,1]$ by setting  
\begin{equation}\label{defkappa0}
    \zeta_{N}^{{\rm{min}}} (x)   : =  
         \chi^{{\rm{min}}} ( \overline x )     \   \     ,   \  \  
         \zeta^0_N(x) :=   \chi_0(\overline x)
       \    \   ,      \     \   \tilde\zeta_{N}(x)    :=   
  \big(1-[\zeta^{{\rm{min}}}_{N}]^2 (x)
  -  [\zeta^{0}_{N}]^2 (x)  \big)^\frac12    , 
 \end{equation}
 so that $  [\zeta^{{\rm{min}}}_{N}]^2
 + [\zeta^{0}_{N}]^2 + [\tilde \zeta_{N}]^2 \equiv 1$. 
Note that $\zeta^{{\rm{min}}}_{N},\zeta^{0}_{N}, 
\tilde \zeta_{N}\in C^\infty(\mb R^N)$ 
for all $N\in \mb N$. 

\

\noindent
It follows from the IMS localization formula that for all $N\in \mb N$, $h>0$
 and all $f\in C_b^\infty(\mb R^N)$
\begin{gather}\label{IMSondiagonal}
\mc E_{h, N} [f]   = 
\mc E_{h, N} [\zeta^{{\rm{min}}}_{N}f]  + 
  \mc E_{h, N} [\zeta^{0}_{N}f]  + 
  \mc E_{h, N} [\tilde \zeta_{N}f]  + 
       h \mc G_{h, N}[f]           ,   
\end{gather}
where the localization error $\mc G_{h, N}[f]$ is given by 
 \[ \mc G_{h, N}[f] :=  -   N \int_{\mb R^N} \left( |\nabla \zeta^{{\rm{min}}}_{N}|^2  + 
  |\nabla \zeta^0_{N}|^2    +  |\nabla \tilde\zeta_{N}|^2    \right)  f^2  \, e^{- \frac{V_N }{hN}}  dx 
  .   \]
 Computing the gradients in this formula shows that
  there exists a constant $c>0$ such that 
  for all $N\in \mb N$, $h>0$ and $f\in C_b^\infty(\mb R^N)$
\[    |\mc G_{h, N}[f] |  \leq   c  \int_{\mb R^N}   f^2  \, e^{- \frac{V_N }{hN}}  dx .   \]
Using for the first, second and third term on the right hand side of~\eqref{IMSondiagonal} respectively 
Proposition~\ref{MainPropDiagMinima},
Proposition~\ref{propsaddle} and Proposition~\ref{propinflection} gives then 
the following estimate: 
there exist constants $R_0, \tilde h_0, C_{\rm{min}}, C_0, \tilde C>0$
 and, for every 
$h>0, N\in \mb N$, there exist  functions $\vp^{+}_{h, N},\vp^{-}_{h, N}   
\in C_b^\infty(\mb R^N)$
 such that for all $h\in (0, \tilde h_0]$, $N\in \mb N$ and 
  $f\in C_b^\infty(\mb R^N)$ with    
  $\supp f  \subset \mc S_N(R_0)$ it holds
\begin{gather*}
 \mc E_{h,N} [f]     \geq      C_{\rm{min}}    \int_{\mb R^N}   [\zeta^{{\rm{min}}}_{N} f]^2 e^{-\tfrac{V_N}{hN}}  dx  \\
     -    \Big(\int_{\mb R^N}  f  \, \zeta^{{\rm{min}}}_{N}  \vp^{+}_{h, N}    e^{-\tfrac{V_N}{hN}}  dx \Big)^2 
     -     
      \Big( \int_{\mb R^N}  f  \, \zeta^{{\rm{min}}}_{N} \vp^{-}_{h, N}    e^{-\tfrac{V_N}{hN}}  dx  
       \Big)^2    \\
       +        C_{0}    \int_{\mb R^N}   [\zeta^{0}_{N} f]^2 e^{-\tfrac{V_N}{hN}}  dx     + 
           \tfrac{\tilde C}{h_0}    \int_{\mb R^N}   [\tilde \zeta_{N} f]^2 e^{-\tfrac{V_N}{hN}}  dx   
           - c  h  \int_{\mb R^N}   f^2  \, e^{- \frac{V_N }{hN}}  dx  . 
       \end{gather*}
Taking $C:= \tfrac 12 \min\{C_{\rm{min}}, C_0, \tilde C\}$, 
$h_0 := \min\{\tilde h_0, 1, \tfrac{C}{c}\} $ 
and      $\phi^{\pm}_{h, N} := \zeta^{{\rm{min}}}_{N}  \vp^{\pm}_{h, N}$
conlcudes the proof. 
 \end{proof}

\section{Spectral Gap Asymptotics in infinite dimension}\label{sec.Inf}

\noindent
In this section we provide the proof of Theorem~\ref{corollaryEK}. The latter is split into two parts:  The lower bound will be shown with Proposition~\ref{propLBinfty} below, and the upper Bound with Proposition~\ref{propUBinfty} below.  Both lower and upper bound are obtained by rather straightforward approximation procedures starting from the corresponding uniform result on the lattice. Thus the lower bound given in Proposition~\ref{propLBinfty}
will be essentially a corollary of Theorem~\ref{thKramers}.

\

\noindent
We start by introducing our notation. 
We fix $m>0$ and consider  for each $h>0$
the  centered Gaussian measure $\gamma_h$ 
on $L^2(\mb T)$ introduced after Theorem~\ref{thKramers}. The Fourier transform 
of $\xi\in L^2(\mb T)$ is denoted by  
$k\mapsto \hat \xi(k) : = \int_{\mb T} \xi(s) e^{- i2\pi ks} \, ds $.
For the approximation procedure it will be enough and notationally convenient to consider only odd $N$.
We shall then denote by $\mb N':=  2\mb N -1$ the
set of odd natural numbers and label generic 
elements of $\mb R^{N}$ for $N\in \mb N'$
 by $x = (x_k)_{k\in \mb Z: |k|\leq \tfrac {N-1}{2}}$. We still assume periodic boundary conditions, i.e. $x_{\frac 12 (N+1)} := x_{ - \frac 12 (N - 1)}  $. 

\

\noindent
For each $N\in \mb N'$ and $h>0$ we consider the  centered Gaussian probability measure $\gamma_{h,N}$
on $\mb R^N$ defined by
\[       \gamma_{h, N}(dx)    :=    Z^{-1}_{h, N} \exp{   \left( -  \frac m2 \sum_{k} x_k^2  -   \frac{\mu}{8\sin^2(\frac{\pi}{N})}   \   \sum_{k}  
   (x_{k} - x_{k+1})^2  \right)}       \, dx  ,       \]
   where $Z_{h, N}$ is the normalization constant. 
Note that for every $N\in \mb N', h>0$ and $k, j\in \mb Z$ with $|k|, |j| \leq \frac 12 (N-1)$
 the covariances are given by 
\begin{equation}\label{covariancegamma2} 
 \int_{\mb R^{N}}   x_k x_j   \,  \gamma_{h, N} (dx) 
  =    h \sum_{|\ell|\leq \frac 12(N-1)}   \frac{1}{m+ \nu_{\ell, N}} \, \exp{\left(- \im 2\pi \tfrac{\ell}{N} (k-j)\right)}  ,     
  \end{equation}
  where the $\nu_{\ell, N}$'s are the eigenvalues of $K_{N}$
as introduced in~\eqref{eigenvaluesK}.
In order to relate $\gamma_{h}$ and $\gamma_{h, N}$ we 
introduce for each  $N\in \mb N'$ a random vector 
$X^{(N)} = (X_k^{(N)})_{k\in \mb Z: |k|\leq \tfrac {N-1}{2}}$
 on the probability space
$(L^2(\mb T), \gamma_h)$
which has the property to be distributed according to $\gamma_{h, N}$. 
This can be done as follows~\cite{Si_EQFT}. We define for each $N\in \mb N'$ the function $\sigma := \sigma_N: \mb Z \to \mb R$ by setting 
 \begin{equation}\label{defMN}
 \sigma_N(k)  :=    \begin{cases}
  \left( \frac{m+ \mu k^2}{m + \nu_{k, N}}\right)^{\frac 12}  &     \text{if } |k|=0, \dots, \tfrac{N-1}{2} ,    \\
    0   & \text{otherwise. }
 \end{cases}
 \end{equation}
Note that $\lim_{N\to \infty}\sigma_N(k) \to 1$ for every $k\in \mb Z$ and that 
there exists a $C>0$ such that 
$|\sigma_N(k)| \leq C$ for all $k\in \mb Z$ and $N\in \mb N'$.  
Further  let $X^{(N)} = (X_k^{(N)})_{k\in \mb Z: |k|\leq \tfrac {N-1}{2}}: L^2(\mb T) \to \mb R^{N}$ be the continuous linear map  given by  
  \begin{equation}\label{defX}    X_k^{(N)} (\xi)    :=    \int_{\mb T} \xi(s)  \sum_{\ell\in \mb Z} 
  \sigma_N(\ell)  \, e^{ - \im 2\pi \ell (s-\tfrac kN)  }   \, ds
  \  \    ,    \  \  \forall \xi \in L^2(\mb T) , |k| \leq \tfrac{N-1}{2}.  
  \end{equation}
\noindent
Then $X^{(N)}$ is a centered Gaussian random vector on the probability space $(L^2(\mb T), \gamma_h)$ with covariance given by 
\[   \int_{L^2(\mb T)}   X_k^{(N)}  X_j^{(N)}  \, \gamma_h(d\xi)    =    h  \sum_{|\ell|\leq \frac 12( N-1)}    \frac{1}{m+ \nu_{\ell, N}} \, \exp{\left(- \im 2\pi \tfrac{\ell}{N} (k-j) \right)}     
  . \] 
 Thus, it follows from~\eqref{covariancegamma2}
 that $X^{(N)}$ has Law $\gamma_{h, N}$ as desired. Furthermore, we introduce in analogy to [DaPrato, 11.2], point evaluation on $L^2(\mb T)$ for all $ s\in \mb T$ and $\xi \in L^2(\mb T)$ as $\xi(s):=\delta_s(\xi)$, where $\delta_s$ is defined as the limit in $L^2(d\gamma_h)$ of
   \begin{equation}\label{defX}    \delta_{s,N} (\xi)    :=   \sum_{|\ell|\leq \frac 12(N-1)} 
  \hat \xi(\ell)e^{ \im 2\pi \ell s  }.
  \end{equation} 
  
 \
 
 \noindent
 Next, recalling the functional $U$ introduced in~\eqref{defUc}, we consider for each $N\in \mb N'$ its discrete version $U_N:\mb R^N \to \mb R$ given by 
 \begin{equation}\label{defUdiscr}     U_N(x)      :=           \sum_{k}   \left(      \frac 14 x_{k}^4  -  \frac {1}{2}(1+m) x_k^2   + \frac 14       \right)         .   
   \end{equation}

\noindent
We shall consider for each $h>0$ and $N\in \mb N'$ the perturbed probability measures $\tilde \gamma_{h}$ and $\tilde \gamma_{h, N}$, defined 
respectively on the Borel sets of $L^2(\mb T)$ and  $\mb R^N$, and given by 
 \[ \tilde \gamma_{h}(d\xi)   :=  \tilde Z_h^{-1}e^{-\tfrac{U(\xi)}{h}} \, \gamma_h(d\xi)    \  \  \   ,   \   \  \  \   
 \tilde \gamma_{h, N}(dx)   :=   \tilde Z_{h, N}^{-1}  e^{-\tfrac{U_N(x)}{hN}} \, \gamma_{h, N}(dx)     ,    \]
 where $\tilde Z_h$ and $\tilde Z_{h, N}$ are normalization constants. 
\begin{remark}  An explicit Gaussian computation shows that $\xi\mapsto \int_{\mb T} \xi^4(s) \, ds$ is in $L^1(d\gamma_h)$ for every $h>0$. This implies, by Jensen's inequality, $\tilde Z_h >0$ and thus $\tilde \gamma_{h}$ is well-defined. 
It is well-known that 
not only $\gamma_h(L^4(\mb T)) =1$, but even $\gamma_h(C^{\alpha}(\mb T)) =1$ for 
$\alpha< \tfrac 12$, where $C^{\alpha}(\mb T)$ is the space of $\alpha$-H\"older continuous functions on $\mb T$.
\end{remark}

\noindent
Now fix $F\in \mc FC_b^\infty(L^2(\mb T))$. Then by definition 
of $\mc FC_b^\infty(L^2(\mb T))$
 there exist
$n\in \mb N$, $\mbf y= (y_1, \dots, y_n)$ with $y_i\in L^2(\mb T)$ and $f\in C_b^\infty(\mb R^n)$ such that $F(\xi) =   
f( [\xi, \mbf y])$ for every $\xi\in L^2(\mb T)$, 
where we have set for short 
\[[\xi, \mbf y]    := \left(   \lp \xi,y_1 \rp_{L^2(\mb T)} , \dots, 
 \lp \xi, y_n \rp_{L^2(\mb T)}   \right).  \] 
In particular  $F\in C^1(L^2(\mb T))$ and, 
denoting by $D F:
L^2(\mb T) \to L^2(\mb T) $ the gradient of $F$, we have 
$\forall \xi \in 
L^2(\mb T)$ 
\begin{equation}\label{representationD}    \|D F(\xi)\|^2_{L^2(\mb T)}   =   \sum_{j,\ell=1}^n
\partial_{j} f([\xi, \mbf y])   \, 
 \partial_{l} f([\xi, \mbf y])  \lp y_j, y_\ell \rp_{L^2(\mb T)}    .  
 \end{equation}
In order to perform the approximation procedure it will be convienent to introduce for each $N\in \mb N'$
the function $\tilde f_N:\mb R^{N} \to \mb R$ defined 
by
\begin{equation}\label{ftilde}
\tilde f_N(x)   : =     F(\eta^{(N)} (x) )   =   f([\eta^{(N)}( x), \mbf y] ), 
\end{equation}
where  $\eta^{{N}} : \mb R^{N} \to L^2(\mb T)$ and 
$ \eta^{(N)} (x) $ is 
 is given by 
 \begin{equation}\label{defeta}
  s\mapsto  \frac 1N \sum_{|j|, |\ell|\leq \frac 12(N-1)}     x_{\ell}  \, e^{ i2\pi j(s-\tfrac \ell N)}  .  
 \end{equation}
 \noindent
Before we can prove the following approximation result, we introduce the covariance functions
\begin{itemize}
\item [] $Q(s,t):=\int_{L^2(\mb T)} \xi(s) \xi(t) \, \gamma_h(d\xi)$,
\item [] $Q^N(s,t):=\int_{L^2(\mb T)}  X_{k_s}^{(N)} (\xi) X_{k_t}^{(N)} (\xi) \, \gamma_h(d\xi)$,
\item [] $\tilde Q^N(s,t):=\int_{L^2(\mb T)} X_{k_s}^{(N)}(\xi)\xi(t) \, \gamma_h(d\xi)$,
\end{itemize}
defined for all $s,t\in \mb T$ where $k_s:=\text{int}(N\cdot s)$ is the next lattice point to $s$. Note that $Q^N$ is just the step function associated to the discrete covariance $\int_{L^2(\mb T)}  X_{k}^{(N)} (\xi) X_{\ell}^{(N)} (\xi) \, \gamma_h(d\xi)$. Using the boundedness and convergence of $\sigma_N$, a simple argument shows that $Q^N,\tilde Q^N$ are uniformly bounded and $Q=\underset{N\rightarrow \infty}{\lim}Q^N=\underset{N\rightarrow \infty}{\lim}\tilde Q^N$ pointwise. Furthermore, we recall that we can express higher momenta of Gaussian random variables via a polynomial of their covariances. Particularly, there exists a polynomial $p(x,y,z)$ such that $\int_{L^2(\mb T)}   (\frac{1}{4}\phi^4-\frac{1}{2}(m+1)\phi^2) (\frac{1}{4}\psi^4-\frac{1}{2}(m+1)\psi^2) \, \gamma_h(d\xi)$ can be expressed by 
\begin{align*}
p\left(\int_{L^2(\mb T)}   \phi\cdot \phi \, \gamma_h(d\xi),\int_{L^2(\mb T)}   \psi\cdot \psi \, \gamma_h(d\xi),\int_{L^2(\mb T)}   \phi\cdot \psi \, \gamma_h(d\xi)\right),
\end{align*}
for all Gaussian distributed random variables $\phi,\psi$.

\begin{lemma} \label{LemmaapproxU}. 
Let $U_N$, $X^{(N)}$ and $\eta^{(N)}$ be defined respectively as
in~\eqref{defUdiscr},~\eqref{defX} and~\eqref{defeta}.
 Then, in the limit $N\to \infty$,  we have
\begin{itemize}
\item[(i)]  
$\tfrac{U_{N}\circ X^{(N)}}{N}  \to    U(\xi)  $
in $L^2(d\gamma_h)$, 
\item[(ii)]    $\lp \eta^{(N)}(X^{(N)}(\xi)), y \rp_{L^2(\mb T)}   \to   \lp \xi  , y \rp_{L^2(\mb T)} $   for all $\xi, y\in L^2(\mb T)$.
\end{itemize}
\end{lemma}

\begin{proof} 
For a series $\{a_k:\ |k|\leq \frac 12( N-1)\}$, we have the identity $\frac{1}{N}\sum_{|k|\leq \frac 12( N-1)} a_k=\int_{\mb T} a_{k_t}dt$, hence we obtain
\begin{align*}
&\int_{L^2(\mb T)}   \left|U(\xi)- \tfrac{U_{N}( X^{(N)} (\xi))}{N}\right|^2 \, \gamma_h(d\xi)=\\
=&\int_{L^2(\mb T)}\left|\int_{\mb T}  \frac 14  \xi^4(s)-\frac{1}{2}(m+1)\xi^2(s)-\frac{1}{4}\Big(X_{k_s}^{(N)}(\xi)\Big)^4+\frac{1}{2}(m+1)\Big(X_{k_s}^{(N)}(\xi)\Big)^2ds\right|^2  \, \gamma_h(d\xi)
\end{align*}
Rewriting this expression in terms of the covariances $Q,Q^N,\tilde Q^N$ and polynomial $p(x,y,z)$, yields
\begin{align*}
\int_{L^2(\mb T)}\int_{L^2(\mb T)}&\Big(p\big(Q(s,s),Q(t,t),Q(s,t)\big)-2p\big(Q(s,s),Q^N(t,t),\tilde Q^N(s,t)\big)\\
&+p\big(Q^N(s,s),Q^N(t,t),Q^N(s,t)\big)\Big)ds \ dt,
\end{align*}
where we have expanded the square and changed the order of integration. Because of $Q=\underset{N\rightarrow \infty}{\lim}Q^N=\underset{N\rightarrow \infty}{\lim}\tilde Q^N$, statement (i) follows by dominated convergence.
 In order to prove (ii) note that
\begin{gather*}
\lp \eta^{(N)}(X^{(N)}(\xi)), y \rp_{L^2(\mb T)}    =     
\frac 1N \int_{\mb T}
\sum_{|j|, |\ell|\leq \frac 12(N-1)} X_{\ell}^{(N)}(\xi) \,  e^{\im 2\pi j(s- \frac \ell N))}   y(s) \, ds     =   \\
\frac 1N \sum_{|j|, |\ell|, |p|\leq \frac 12(N-1)}  e^{-\im 2\pi \frac \ell N(j+p)}  
\sigma_N(p) \hat \xi(-p) \hat y(-j)     = \\
 \sum_{|j|\leq \frac 12(N-1)}   \sigma_N(j) \hat \xi(j) \hat y(-j)  
\xrightarrow{N\to \infty} 
 \sum_{j\in \mb Z}    \hat \xi(j) \hat y(-j)  
   =       \lp  \xi, y \rp_{L^2(\mb T) }  .    
 \end{gather*}

\end{proof}

\noindent
The next lemma provides for fixed $h>0$ the crucial approximation properties for the objects appearing in the functional inequality which defines the spectral gap $\lambda_h^{(1)}$. 
\begin{lemma}\label{PropapproxDir} Let $F\in \mc FC_b^\infty(L^2(\mb T))$
and consider for each $N\in \mb N'$ the corresponding $\tilde f_N \in C_b^\infty( \mb R^{N})$ defined in~\eqref{ftilde}. Then 
for each $h>0$
\begin{equation}\label{approxDirForm}
\lim_{N\to \infty}
\left| \int_{L^2(\mb T)}    \|D F(\xi)\|^2_{L^2(\mb T)} \, 
 \tilde\gamma_h(d\xi)    -      
N \int_{\mb R^{N}}   |\nabla \tilde f_N (x)|^2 \, 
\tilde\gamma_{h, N}(dx)  \right|     =   0     .
\end{equation}
Moreover, for each $h>0$, we have 
 \begin{itemize}
 \item[]$\lim_{N\to \infty}\left|\int_{L^2(\mb T)}    |F(\xi)|^2 \, 
 \tilde\gamma_h(d\xi)- \int_{\mb R^{N}}   |\tilde f_N (x)|^2 \, 
\tilde\gamma_{h, N}(dx)\right|=0$,
\item[]$\lim_{N\to \infty}\left|\int_{L^2(\mb T)}    F(\xi) \, 
 \tilde\gamma_h(d\xi)- \int_{\mb R^{N}}   \tilde f_N (x) \, 
\tilde\gamma_{h, N}(dx)\right|=0$.
 \end{itemize}
\end{lemma}

\begin{proof} We shall prove only~\eqref{approxDirForm}, 
since the statements (i) and (ii) are proven similarly, with even some simplifications. 
Fix $h>0$ and $F\in \mc FC_b^\infty(L^2(\mb T))$. Then 
\begin{gather*}
 \int_{L^2(\mb T)}    \|D F(\xi)\|^2_{L^2(\mb T)} \, 
 \tilde\gamma_h(d\xi)    -      
N \int_{\mb R^{N}}   |\nabla \tilde f_N (x)|^2 \, 
\tilde\gamma_{h, N}(dx) 
    =     \\
     \mc R_1(N)      + \mc R_2(N), 
\end{gather*}
where
\[\mc R_1(N)   : =   \int_{L^2(\mb T)}   \left[ \tilde Z_h^{-1}\|D F(\xi)\|^2_{L^2(\mb T)} 
- \tilde Z_{h,N}^{-1} N |\nabla \tilde f_N (X^{(N)} (\xi))|^2
\right] e^{-\tfrac{U(\xi)}{h}}\, \gamma_h(d\xi)     ,   \]
\[ \mc R_2(N) :=  \tilde Z_{h,N}^{-1}\int_{L^2(\mb T)}    N |\nabla \tilde f_N (X^{(N)} (\xi))|^2 \, \left[e^{-\tfrac{U(\xi)}{h}}  -   e^{- \tfrac{U_{N}( X^{(N)} (\xi))}{hN}}  \right] \, \gamma_h(d\xi)   . \]
In order to prove 
the convergence of the first term, we shall show that 
\begin{equation}    \label{pointwiseconvtildef}
\lim_{N \to \infty} N |\nabla \tilde f_N (X^{(N)} (\xi))|^2
   =   \|D F(\xi)\|^2_{L^2(\mb T)}       \   \  \    \  \  \  
   \text{ for all }   \xi \in L^2(\mb T)          .
\end{equation}
The convergence $\mc R_1(N) \to 0$ will then follow by dominated convergence. 
 To prove~\eqref{pointwiseconvtildef} observe that
 for every $x\in \mb R^{N}$ we have
 \begin{gather*}
 | \nabla\tilde f_N(x)|^2    =  \\
 \sum_{j, \ell=1}^n \partial_j  
 f ([\eta^{(N)}( x), \mbf y] )   
 \, \partial_\ell f ([\eta^{(N)}( x), \mbf y] )  \, \frac {1}{N^2} \sum_{|k|\leq \frac 12(N-1)} \tilde y_j(k)
 \tilde y_\ell(k)   ,
\end{gather*}
where for generic $y\in L^2(\mb T)$, every $N\in \mb N'$ and $k\in \mb Z:|k| \leq \frac 12(N-1)$ we define 
 $\tilde y (k)= \tilde y^{(N)} (k): =\sum_{|p|\leq \frac 12(N-1)}
 \int_{\mb T} y(s) e^{i2\pi p(s- \frac kN)} \, ds $. 
It follows that for every $\xi\in L^2(\mb T)$ we have
\begin{gather*} N |\nabla \tilde f_N (X^{(N)} (\xi))|^2   =  \\
 \sum_{j, \ell=1}^n \partial_j  
 f ([\eta^{(N)}( X^{(N)} (\xi)), \mbf y] )   
 \, \partial_\ell f ([\eta^{(N)}( X^{(N)} (\xi)), \mbf y] )  
 \frac 1N\sum_{|k|\leq \frac 12(N-1)} \tilde y_j(k)
 \tilde y_\ell(k)   .     
 \end{gather*}
 Statement~\eqref{pointwiseconvtildef}  follows then from the convergence result of Lemma~\ref{LemmaapproxU} (ii),
 the representation of $ \|D F(\xi)\|^2_{L^2(\mb T)} $ given in~\eqref{representationD}
and the following computation: 
\begin{gather*}
\frac 1N\sum_{|k|\leq N-1} \tilde y_j(k)
 \tilde y_\ell(k)  =  \frac 1N\sum_{|k|\leq \frac 12(N-1)}
 e^{i2\pi \frac kN(q+p)}
 \sum_{|p|,|q|\leq \frac 12(N-1)} 
  \hat y_j(-p) 
    \hat  y_{\ell}(-q)    = 
       \\
       \sum_{|p|\leq \frac 12(N-1)} 
  \hat y_j (p) 
    \hat y_{\ell}(-p)    \xrightarrow{N\to \infty}     \sum_{p\in \mb Z} 
  \hat y_j (p) 
    \hat y_{\ell}(-p)     =    \int_{\mb T}    y_j(s) y_k(s) \, ds.  
\end{gather*}
Note that $N |\nabla \tilde f_N (X^{(N)} (\xi))|^2$ is uniformly bounded for $F\in \mc FC_b^\infty(L^2(\mb T))$. Regarding the convergence of the second term $\mc R_2(N)$, we use this fact together with Lemma~\ref{LemmaapproxU} (i) to obtain $\tilde Z_{h,N}\rightarrow \tilde Z_{h}$ as well as $\mc R_2(N) \to 0$.
 \end{proof}

\noindent
We can now prove the lower bound in the statement of Theorem~\ref{corollaryEK}.
\begin{proposition}[Asymptotic lower bound on $\lambda_h^{(1)}$] \label{propLBinfty}
There exist $h_0, C>0$ such that
for every $h\in (0, h_0]$ and every $F\in \mc FC_b^\infty(L^2(\mb T))$ satisfying 
$\int_{L^2(\mb T)}    |F(\xi)|^2 \, 
 \tilde\gamma_h(d\xi) = 1$ and 
 $\int_{L^2(\mb T)}    F(\xi) \, 
 \tilde\gamma_h(d\xi) = 0 $, 
 we have 
 \begin{equation*}
 h \int_{L^2(\mb T)}    \|D F(\xi)\|^2_{L^2(\mb T)} \, 
 \tilde\gamma_h(d\xi) 
  \geq 
    \tfrac{ \sinh (\pi \sqrt{2\mu^{-1}})}{ \pi  \sin (\pi \sqrt{\mu^{-1}})}  \    
e^{-\frac{1}{4h}}
\  \big(    1   - C h      \big)          .
\end{equation*} 
\end{proposition}
\begin{proof} 
Take $h_0>0$ as in the statement of Proposition~\ref{thKramers}. Fix $h\in (0, h_0]$ and $F\in \mc FC_b^\infty(L^2(\mb T))$ such that  
$\int_{L^2(\mb T)}    |F(\xi)|^2 \, 
 \tilde\gamma_h(d\xi) = 1$ and 
 $\int_{L^2(\mb T)}    F(\xi) \, 
 \tilde\gamma_h(d\xi) = 0 $. It follows then 
 from Lemma~\ref{PropapproxDir} that 
  \begin{gather}
 h \int_{L^2(\mb T)}    \|D F(\xi)\|^2_{L^2(\mb T)} \, 
 \tilde\gamma_h(d\xi)   =   h N \int_{\mb R^{N}}   |\nabla \tilde f (x)|^2 \, 
\tilde\gamma_{h, N}(dx)    +    o(1)   \geq  \label{approxdir}  \\
 \lambda^{(1)}_{h, N}  \left( 1 +   o(1)   \right)     +    o(1) ,
 \nonumber 
 \end{gather}
 where $o(1)$ denotes a possibly on $h$ dependent sequence which vanishes in the limit $N\to \infty$. 
  Proposition~\ref{thKramers} implies that 
 there exists $C>0$ such that for all $N\in \mb N'$
 \[     \lambda^{(1)}_{h, N}   \geq  p(N)   \    
e^{-\frac{1}{4h}}
\  \big(    1   - C h      \big)   , \]
which together with~\eqref{approxdir} gives  
\begin{equation}\label{lbdirichletN}  h \int_{L^2(\mb T)}    \|D F(\xi)\|^2_{L^2(\mb T)} \, 
 \tilde\gamma_h(d\xi)     \geq 
 p(N)   \    
e^{-\frac{1}{4h}}
\  \big(    1   - C h      \big)     \left( 1 +   o(1)   \right)     +    o(1)   .  
\end{equation}
Passing to the limit $N\to \infty$ on the right hand side  of~\eqref{lbdirichletN} and recalling~\eqref{prefactorconv} finishes the proof. 
\end{proof}
\noindent
The proof of Theorem~\ref{corollaryEK} is completed by the following proposition. 

\begin{proposition}[Asymptotic upper bound on $\lambda_h^{(1)}$]\label{propUBinfty}
There exist $h_0, C>0$ such that
for every $h\in (0, h_0]$ we have
\begin{equation}\label{UBSGinftydim}
 \lambda_h^{(1)} 
  \leq 
    \tfrac{ \sinh (\pi \sqrt{2\mu^{-1}})}{ \pi  \sin (\pi \sqrt{\mu^{-1}})}  \    
e^{-\frac{1}{4h}}
\  \big(    1   + C h      \big)          .
\end{equation} \end{proposition}
\begin{proof}
We consider for each $h>0$ and 
$N\in \mb N'$ the $\mc FC_b^\infty(L^2(\mb T))$ function $\chi_h(\xi):=\chi_{h,0}( \int_{\mb T} \xi(s)   \, ds)$, where we define
\begin{gather*}
\chi_{h,0}(x):=\frac{2}{\sqrt{2\pi h}}\int^x_0 e^{-\frac{s^2}{2h}}   \, ds.
\end{gather*}
Because of symmetry, we obtain $\int_{L^2(\mb T)} \chi_h(\xi)   \, \tilde\gamma_h(d\xi)=0$, and therefore the proof is complete, if we can show
\begin{gather*}
h \int_{L^2(\mb T)}    \|D \chi_h(\xi)\|^2_{L^2(\mb T)} \, \tilde\gamma_h(d\xi)\leq \tfrac{ \sinh (\pi \sqrt{2\mu^{-1}})}{ \pi  \sin (\pi \sqrt{\mu^{-1}})}  \    
e^{-\frac{1}{4h}}
\  \big(    1   + C h      \big)
\int_{L^2(\mb T)}    |\chi_h(\xi)|^2 \, \tilde\gamma_h(d\xi).
\end{gather*}
A straightforward computation yields $\widetilde{(\chi_h)}_N:=\chi_h(\eta^{(N)} (x) )=\chi_{h,0}(\overline x)$, which are exactly the approximate eigenfunctions $\chi_{h,N}$ introduced in Definition 3.10 in~\cite{DGLP}. Moreover it was shown in~\cite[Lemma 3.12 and Prop. 3.13]{DGLP} by computing the relevant Laplace asymptotics that for all $N\in \mb N'$ and $h\in (0, 1]$ 
\begin{gather*}
h N\int_{\mb R^{N}}   |\nabla \chi_{h,N }(x)|^2 \, 
\tilde\gamma_{h, N}(dx)    \leq 
   p(N)   \    
e^{-\frac{1}{4h}}
\  \big(    1   + C h      \big)\int_{\mb R^{N}}   |\chi_{h,N }(x)|^2 \, 
\tilde\gamma_{h, N}(dx).
\end{gather*}
Together with ~\eqref{prefactorconv} and Lemma~\ref{PropapproxDir}, this gives us the desired statement.   
\end{proof}

\noindent{\bf Acknowledgements:}   GDG gratefully acknowledges the financial support of  HIM Bonn in the framework of the 2019 Junior Trimester Programs ``Kinetic Theory'' and ``Randomness, PDEs and Nonlinear Fluctuations'' and the hospitality at the University of Rome La Sapienza during his frequent visits.

\end{document}